\documentclass[12pt]{article}
\usepackage[english]{babel}
\usepackage{amsmath,amsthm,amsfonts, amssymb,epsfig, titlesec, latexsym, amscd}
\usepackage[left=1in,top=1in,right=1in]{geometry}
\usepackage[mathscr]{eucal}
\usepackage{caption}
\usepackage{subcaption}
\usepackage{hyperref}
\usepackage{float}

\usepackage{authblk}
\usepackage{fancyhdr}

\numberwithin{equation}{section}

\newtheorem{thm}{Theorem}

\newtheorem{lem}{Lemma}[section]
\newtheorem{prop}[lem]{Proposition}
\newtheorem{cor}[lem]{Corollary}

\newtheorem{problem}{Problem}

\newtheorem{conjecture}{Conjecture}

\theoremstyle{definition}

\newtheorem{rem}[lem]{Remark}

\titleformat*{\section}{\normalsize \bfseries \filcenter}
\titleformat*{\subsection}{\normalsize \bfseries \filcenter}

\newcommand{\eps}{\varepsilon}

\DeclareMathOperator{\arccosh}{arccosh}

\newcommand{\Addresses}{{% additional braces for segregating \footnotesize
  \bigskip
  \footnotesize

  \textsc{Department of Mathematics, The Pennsylvania State University, University Park, PA}\par\nopagebreak
  \textit{E-mail address}: \texttt{axe930@psu.edu}
	
	\bigskip
	
	\textsc{Department of Mathematics, The Pennsylvania State University, University Park, PA}\par\nopagebreak
  \textit{E-mail address}: \texttt{axk29@psu.edu}
}}

\begin{document}

\title{\normalsize \textbf{ Flexibility~of~entropies~for~surfaces~of~negative~curvature}}
\author{ \normalsize Alena Erchenko and Anatole Katok}
\date{}
\maketitle
\thispagestyle{empty} %removes page number
\begin{abstract}
We consider a smooth closed surface $M$ of fixed genus $\geqslant 2$ with a Riemannian metric $g$ of negative curvature with fixed total area. The second author has shown that the topological entropy of geodesic flow for $g$ is greater than or equal to the topological entropy for the metric of constant negative curvature on $M$ with the same total area which is greater than or equal to the metric entropy with respect to the Liouville measure of geodesic flow for $g$. Equality holds only in the case of constant negative curvature. We prove that those are the only restrictions on the values of topological and metric entropies for metrics of negative curvature.
\end{abstract}

\tableofcontents

\section{Introduction}

\subsection{Formulation of the result} Throughout this paper we consider a smooth closed orientable Riemannian surface $M$ of fixed genus~$G\geq2$.

Let $g$ be the Riemannian metric on $M$ of negative curvature. The metric $g$ generates an area element on $M$ (Riemannian area form $dA_g$). We denote the total area of $M$ with respect to this area form as $v_g$.  We will fix the area $V$ of $M$ so that $v_g = V$ for all $g$ that we consider. 

For each such metric $g$ on $M$ let  $\mu_g$ be the normalized Riemannian measure  on $M$ which assigns to every set its area divided by $v_g$. Let  $S^gM$ be unit tangent bundle for $g$, i.e. the submanifold of $TM$ consisting of vectors with length one. There is a canonically defined Riemannian metric on $S^gM$ and thus, a normalized Riemannian measure $\lambda_g$ on $S^gM$ called the Liouville measure. The geodesic flow~$\phi^g~=~\{\phi^g_t\}_{t \in \mathbb{R}}$ generated by $g$ is a one-parameter family of diffeomorphisms of $S^gM$ determined by the motion of tangent vectors with unit speed along geodesics corresponding to these tangent vectors. The geodesic flow preserves the measure $\lambda_g$ on $S^gM$. Therefore, we can define the topological entropy $h_g$ of the geodesic flow $\phi^g_t$ and the metric entropy $h^\lambda_g$ of $\phi^g$ with respect to the invariant measure $\lambda_g$. For the sake of  brevity we will often  call the latter quantity  simply {\em metric entropy}. By the Variational Principle $h^\lambda_g\le h_g$ and, since the geodesic flow is an Anosov flow,
\begin{equation}\label{entropy-Anosov}
h^\lambda_g>0.
\end{equation}

For any metric $g$ of constant negative curvature $K$, 

\begin{equation}\label{entropy-cc}
h_g=h_g^{\lambda}=\sqrt {-K} =\left(\frac{4\pi(G-1)}{V}\right)^{\frac{1}{2}}.
\end{equation}

It is proved in  \cite{K82} 
that for every Riemannian metric $g$ of variable (non-constant) negative curvature  total area  $V$ on $M$, 
\begin{equation}\label{entropy-estimate}
h_g^{\lambda}<\left(\frac{4\pi(G-1)}{V}\right)^{\frac{1}{2}}<h_g.
\end{equation}

It the present paper we show that \eqref{entropy-Anosov} and the dyhothomy between \eqref{entropy-cc} and \eqref{entropy-estimate}  are the only restrictions on the values of the topological and metric entropies for the geodesic flow on a Riemannian surface of negative curvature and fixed area. 

\begin{thm} \label{mainthm}
Suppose $M$ is a closed orientable surface of genus $G\geq 2$ and $V>0$. For any $a, b$ such that $a>\left(\frac{4\pi(G-1)}{V}\right)^{\frac{1}{2}}>b>0$, there exists a smooth metric $g$ of negative curvature such that $a = h_g$, $b = h_g^\lambda$ and $v_g = V$.
\end{thm}

The set of all possible values for pairs of entropies is shown in Figure~\ref{fig:ent_set} (shaded area plus its lower right corner). Let us denote this set by $\mathcal D$.

\begin{figure}[H]
\centering
\begin{subfigure}{.5\textwidth}
  \centering
  \includegraphics[height=6cm]{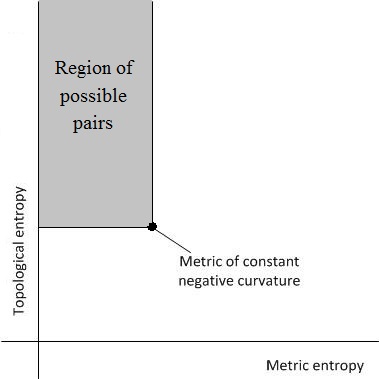}
  \caption{Possible pairs of values of entropies}
  \label{fig:ent_set}
\end{subfigure}%
\begin{subfigure}{.5\textwidth}
  \centering
  \includegraphics[height=6cm]{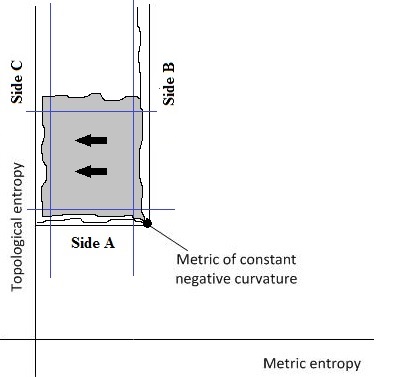}
  \caption{Constructions from the proof}
  \label{fig:constr_set}
\end{subfigure}
\caption{}
\label{fig:entropies}
\end{figure}

\begin{rem} Our result is an  illustration of the {\em flexibility principle} in dynamics. In a rather vague way this principle (or paradigm)  can be formulated as follows:\medskip

\noindent  $(\mathfrak F )$.{\em Under properly understood general restrictions  within a fixed class of smooth dynamical systems, dynamical invariants, both quantitative and qualitative,  take  arbitrary values.}
\medskip

\noindent For another recent  flexibility result of a very different nature see \cite{BKRH}. Flexibility is a potentially  big area of dynamics that up to now has been barely touched. We will discuss several aspects of the flexibility program related to geodesic flows  in Section~\ref{section:open}.
\end{rem}
\medskip

\textbf{Acknowledgments.} This work was partially supported by NSF grant DMS 16-02409. The first author would like to thank Federico Rodriguez Hertz and Alexei Novikov for many helpful and valuable discussions from which she learned a lot and Dmitri Burago, Yuri Burago and Anton Petrunin for listening to the geometric arguments of this paper.

\subsection{Basic entropy formulas}
The constant curvature case is the only one where either of the entropies  is expressed by a closed formula. Still there are several useful expressions that, while  using  asymptotic quantities  related to the geodesic flow, help to obtain estimates needed for the proof of Theorem~\ref{mainthm}.
\bigskip

\noindent{\em Metric entropy.}  Let $g$ be a metric on $M$ of non-positive curvature, For any $v\in S^gM$ there is a unique {\em horocyle} passing through  the foot point $p$ of the vector $v$, perpendicular to $v$  so that $v$ points outside it, i.e. in the direction where the horocycle is convex. Let $\kappa(v)$ be the geodesic  curvature of this horocycle at $p$. Then, 

\begin{equation}\label{me-horocycle}
h_g^{\lambda}=\int_{S^{g}M}\kappa(v)d\lambda_g(v).
\end{equation}
 \bigskip

\noindent{\em Topological entropy.}
Let $\tilde M$ be the universal cover of $M$ and $\tilde g$ the lift of the metric $g$ to $\tilde M$. 
For $x\in \tilde M$ and $r> O$  let $B_g(x,r)$ be the ball of radius $r$  centered at $x$ and let 
$V_g(x,r)$ be the area of that ball. Then,

\begin{equation}\label{top-balls}
h_g=\lim_{R\to\infty}\frac{\log V_g(x,R)}{R}.
\end{equation}
The quantity in the right-hand side of \eqref{top-balls} is called {\em the volume entropy} of the metric $g$. 

Now let $\Gamma$ be the fundamental group of $M$  and for $\gamma\in\Gamma$ let $l_\gamma(x)$ be the shortest length of a loop at $x$ that belongs to $\gamma$.  Let  $N_g(x,r)$ be the number of elements  $\gamma\in\Gamma$ such that $l_\gamma(x)\le r$. Then, 
\begin{equation}\label{top-fundamental}
h_g=\lim_{R\to\infty}\frac{\log N_g(x,R)}{R}.
\end{equation}

\subsection {Outline of  proof}

We construct (large) smooth deformations of metrics of constant negative curvature, estimate topological and metric entropies for certain of resulting metrics and use continuity of entropies of Anosov flows in smooth topologies. We refer to Figure~\ref{fig:constr_set}. 

\begin{enumerate}

\item In Section~\ref{constr1}  we describe  a construction  that produces for any given (arbitrary large) number $N$  and any arbitrary small positive number $\delta$ a curve in the set of pairs of values of entropies $\delta$-close to Side B with one end point corresponding to a metric of constant negative curvature and the other end point corresponding to a metric with topological entropy larger than $N$. 

\item The construction of Section~\ref{constr2} allows to produce curves in the set of pairs of values of entropies arbitrarily close to Side A with one end point corresponding to a metric of constant negative curvature and the other end point corresponding to a metric with metric entropy with respect to Liouville measure smaller than any given positive  number. This construction in its exact form is not used in the proof.

\item  Finally,  in Section~\ref{systole_and_polyhedral} we  combine the first and second constructions. This composite construction allows us to take metrics from the first construction with topological entropies in some range of values and modify them in a smooth way to get metrics realizing a curve in the set of pairs of values of entropies arbitrarily close to Side C with topological entropy staying in some fixed range of values. As a result, we are able to construct two-parametric families of metrics with pairs of values of entropies covering any given rectangle within the semi-infinite strip that is the set of possible values for entropies.
\end{enumerate}

Let us give a short sketch of each of these constructions.
\bigskip

\noindent{\em I.Increasing topological entropy and keeping metric entropy close to the critical value  (Section~\ref{constr1}).}  We consider a surface of constant negative curvature with a rotationally invariant ``collar'',  i.e. a cylindrical area around a simple closed geodesic which allows a circle action by isometries preserving the closed geodesic. We show that it is possible to shrink the simple closed geodesic in the collar by modifying the initial metric of constant curvature in some neighborhood of the  closed geodesic inside the collar. There is a relation  between the length of the original closed geodesic,  the length of the shrunk one and the size of the neighborhood  that is needed to maintain negativity of the curvature. It follows from \eqref{top-fundamental} that  the topological entropy of the geodesic flow for negatively curved metrics  can be estimated from below by the  exponential growth rate  of  the number of words in the fundamental group with respect to a properly weighted word metric in the fundamental group. Our construction allows us to increase this  exponential growth rate  by shrinking one closed curve and controlling the length of a curve from a  properly chosen different homotopy class. The rigorous estimates can be found in Section~\ref{top_1}. The total area of the surface after our shrinking procedure changes in a controlled way. The change in area depends on the conformal class of the initial metric of constant negative curvature. Taking the original metric of constant negative curvature  far enough in the  moduli  space, we can make this change of area arbitrarily small. This allows us to show in Section~\ref{metr_1} that the metric entropy of the geodesic flow with respect to the Liouville measure is close to the metric entropy for the initial metric of constant curvature. Normalizing the area does not change the values of the metric and topological entropies significantly. Therefore, we obtain metrics with the desired total area, metric entropy arbitrarily close to the metric entropy for a metric of constant negative curvature, and arbitrarily large topological entropy. The result of the first construction is summarized in Theorem~\ref{thm:incr_top_entropy} and its corollary.
 \bigskip

\noindent{\em II.Decreasing metric entropy and keeping topological entropy close to the critical value (Section~\ref{constr2}).}  The main idea is to approximate the initial metric of constant curvature by a polyhedral metric, i.e., a metric which has constant curvature (negative or zero) on faces and has singularities of negative curvature type (cone points with angles larger than $2\pi$) at vertices of a fine  triangulation of the surface. We show that the polyhedral metric can be smoothed in a way that preserves the condition of non-positive curvature. By Proposition~\ref{prop:approx}, we can approximate metrics of non-positive curvature by metrics of negative curvature. The result of the second construction is summarized in Theorem~\ref{thm:decr_metr_entropy}.

\bigskip

\noindent{\em III.Decreasing metric entropy and keeping topological entropy large (Section~\ref{systole_and_polyhedral}).}  The third construction is a  combination of the first and second constructions. It is used to vary metric entropy of geodesic flow without significantly changing the topological entropy starting from a metric constructed in Section~\ref{constr1}. The idea is to preserve the closed geodesic that was shortened in the first construction and approximate the metric in $C^0$ topology by polyhedral metrics outside of a collar around  the closed geodesic. As in the second construction, we smooth the polyhedral metric around conical points. A technically subtle point is to  properly describe a smoothing procedure at the gluing of the initial metric on the collar to the polyhedral metric. To make metric entropy arbitrarily small, it is necessary to start from metrics which were constructed by modifying  metrics of constant curvature that are far enough in the moduli space. In Sections~\ref{metr_3} and \ref{top_3}, we estimate the values of metric  and topological entropies for the approximation by polyhedral Euclidean metrics outside of the collar with smoothed cone singularities and smoothed gluing. 

The estimate of the topological entropy is based on the fact that for metrics of nonpositive curvature topological entropy is equal to the volume entropy (see \eqref{top-balls}). The latter is obviously uniformly continuous as a function of the  Riemannian metric in $C^0$ topology. The $C^0$ closeness of the family of metrics to the original metric of constant negative curvature is achieved by choosing a sufficiently fine simplicial partition for the polyhedral approximation. 

Since by \eqref{me-horocycle} the metric entropy for metrics of non-positive curvature depends continuously on the metric in $C^2$ topology, for approximating Side A it is sufficient to choose a smooth family of metrics from the constructed smoothed polyhedral metrics that begins at a metric of constant negative curvature and ends at a metric with metric entropy arbitrary close to $0$. Metric entropy is estimated from above by studying the frequency of visits of a typical orbit to small neighborhoods of vertices, the expansion of vectors as a typical geodesic passes through the ``collar" and using the representation of the curvature of horocycles through the solution of the Ricatti equation.

To realize a given point in the set of possible pairs of values for entropies, we take a smooth compact family of metrics from Section~\ref{constr1} and apply the modification from Section~\ref{constr3} to the whole family. The result of these constructions is summarized in Theorem~\ref{thm:var_metr_pres_top} and its corollary. Theorem~\ref{mainthm} follows from Corollary~\ref{coro}.
\bigskip

Notice that the application of Construction II' and Construction II to the originally chosen metric of constant curvature produce different deformations. In the former case  outside of very small neighborhoods of the vertices of the chosen triangulation the absolute value of the curvature is a constant changing along the deformation from the original value all the way to zero. In the latter original value of the curvature is kept in a collar around a closed  geodesic. This collar has small area but fixed length and its scale is greater than that of the triangulation. In terms of entropy values both deformations produce the bottom of  the ``fuzzy  rectangle''  from Figure~\ref{fig:constr_set}. As we said Construction II' as presented in Section~\ref{constr2} is not needed for the proof of Theorem~\ref{mainthm}. It, however presents the method of decreasing metric entropy while  maintaining topological entropy in a pure form and is useful for understanding our approach. 

\section{Construction I}\label{constr1}

In this section, we construct a family of metrics with arbitrary values of topological entropy and with metric entropy close to that of a metric of constant negative curvature.

\begin{thm}\label{thm:incr_top_entropy}
Suppose $M$ is a closed  orientable surface of genus $G\geq 2$ and $V>0$. For any positive number $\delta$, there exists a smooth family of metrics $\{g_t\}_{t\in[0,1]}$ of negative curvature with total area $V$ such that $g_0$ is a metric of constant curvature and the following hold.
\begin{enumerate}
\item $\left(\frac{4\pi(G-1)}{V}\right)^{\frac{1}{2}}-\delta<h^\lambda_{g_t}\leqslant \left(\frac{4\pi(G-1)}{V}\right)^{\frac{1}{2}}$ for all $t\in [0,1]$.
\item $h_{g_1}>N$.
\end{enumerate}
\end{thm}

Since topological entropy varies continuously with the parameter the intermediate value theorem immediately implies

\begin{cor} For any $h\in[\left(\frac{4\pi(G-1)}{V}\right)^{\frac{1}{2}}, N]$ there exists $t\in [0,1]$ such that $h_{g_t} = h.$
\end{cor}

\subsection{Shrinking the systole}
For any sufficiently small positive number $a$ there exists a metric $\sigma_a$ of constant negative curvature and total area $V$ with a simple closed geodesic $\gamma$ of length $a$ and a rotationally invariant collar around this geodesic. Let us give an accurate description of what is meant by a ``collar''.

We  consider normal coordinates for the metric $\sigma_a$ near the geodesic $\gamma$, i.e., a coordinate system $(u,v)$ such that its $u$-lines (i.e. curves $v=const$) are geodesics orthogonal to $\gamma$ parametrized by arc length, and $u$ is the  length parameter along those $u$-lines normalized so that  the  geodesic $\gamma$ is given by $u=0$. The $v$ coordinate is the arc length on $\gamma$ extended along the $u$-lines. Thus, normal coordinates provide a  map of a cylinder  with $(u,v)$ coordinates into $M$. An {\em $s$-collar } around $\gamma$ is the image of the cylinder $|u|\le s$ as long as it is injective. 

The existence of metrics of constant negative curvature of fixed total area  with arbitrary long collars (i.e. $s$-collars with arbitrary large  values of $s$) follows from the description of Teichm\"uller space via Fenchel-Nielsen parameters. The following three pictures visualize this process.

Every surface of negative Euler characteristic can be decomposed into pairs of pants. Figure~\ref{glueing} illustrates this process for genus two.   
\begin{figure}[H]
      \centering
      \includegraphics[height=4cm]{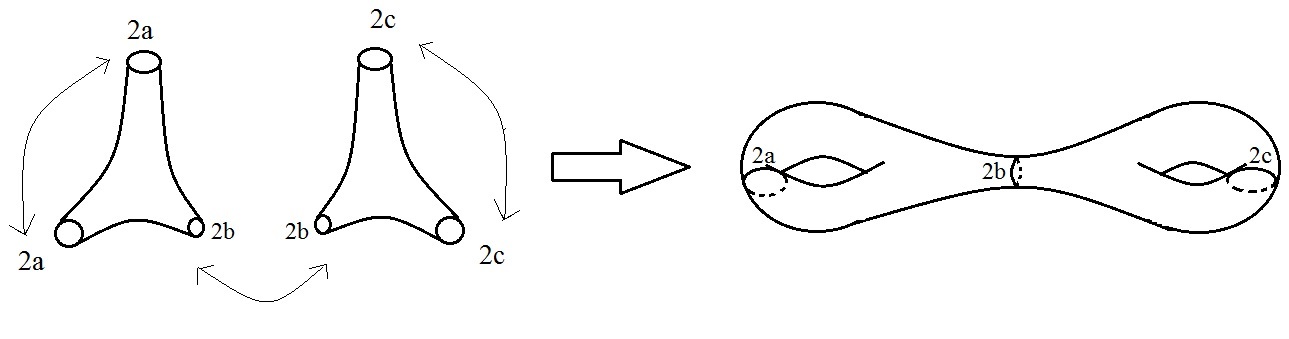}
      \caption{}
      \label{glueing}
      \end{figure}Every pair of pants is constructed by pasting together two copies of a right-angled geodesic hexagon as in Figure~\ref{pants}.
  \begin{figure}[H]
      \centering
      \includegraphics[height=4.5cm]{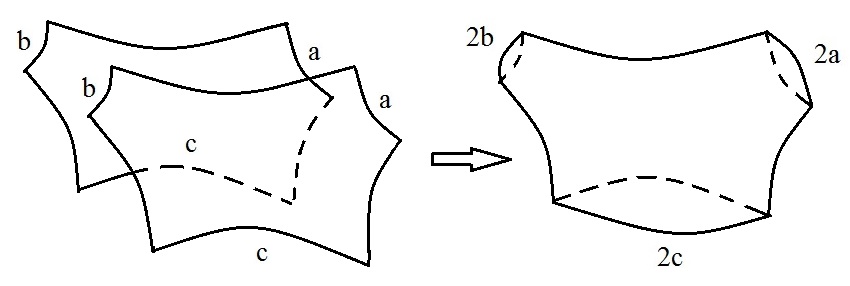}
      \caption{}
      \label{pants}
      \end{figure}
There exists a right-angled geodesic hexagon in the hyperbolic plane with pairwise non-adjacent sides of any prescribed lengths, see Figure~\ref{hexagon}.
  \begin{figure}[H]
      \centering
      \includegraphics[height=5cm]{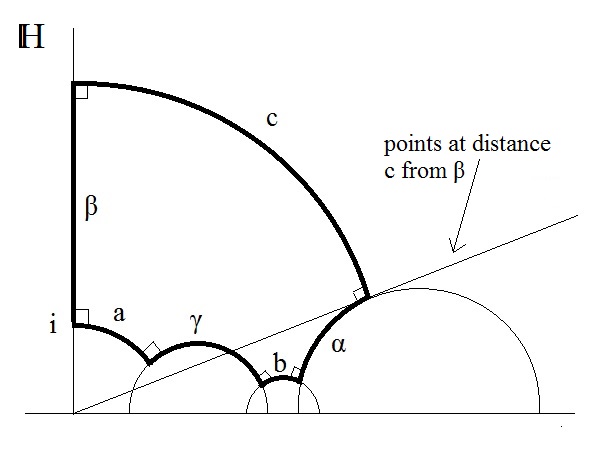}
      \caption{}
      \label{hexagon}
      \end{figure}

The length of the collar around a simple closed geodesic is inversely proportional to the length of the  geodesic. See \cite{B92} for more details.

In  normal coordinates the metric tensor has the form
\begin{equation*}
du^2+ f_a^2(u)dv^2
\end{equation*}
A direct calculation  of the curvature $K(u)$ gives
\begin{equation*}
K(u)= -\frac{f_a''(u)}{f_a(u)}.
\end{equation*}
Thus the curvature is negative if and only if the function $f_a(u)$ is strictly convex and is equal to the constant $K$ if $f_a''(u) = -Kf_a(u)$. Therefore, the curvature is equal to $K$ if  $f_a(u) = a\cosh (\sqrt{-K}u)$ since the collar corresponds to a symmetric solution with respect to the $v$-axis. %(we consider a geodesic ``between" handles).

The surface area of the piece of surface, $D$, corresponding to $u$ varying from $-u_0$ to $u_0$ is equal to
\begin{equation*}
S_{u_0}(a) = \int_0^{2\pi}\int_{-u_0}^{u_0} a\cosh (\sqrt{-K}u)dudv 
= \frac{2\pi a}{\sqrt{-K}} \sinh (\sqrt{-K}u_0).
\end{equation*}

Therefore, $S_{u_0}(a)\rightarrow 0$ as $a\rightarrow 0$.

Now let us show that we can modify the function $f_a$ on a segment of length independent of $a$ and to obtain a metric of negative curvature for which $\gamma$ is still a close geodesic and is  arbitrary short. To do this, we consider the graph of the function $f_a(u)$ in the $(u, y)$-coordinate system and find a tangent line which passes through the origin. The equation of a tangent line through the point $(u_0, f_a(u_0))$ is $y-f_a(u_0) = f'_a(u_0)(u-u_0).$ It passes through the origin if and only if $f_a(u_0) = f'_a(u_0)u_0$, i.e., $\coth(\sqrt{-K}u_0) = \sqrt{-K}u_0$. The equation has a unique pair of symmetric solutions as the function $y(u)=\coth(u)$ is monotone decreasing, the function $y(u)=u$ is monotone increasing, and both functions are centrally symmetric with respect to the origin.

     \begin{figure}
      \centering
      \includegraphics[height=5cm]{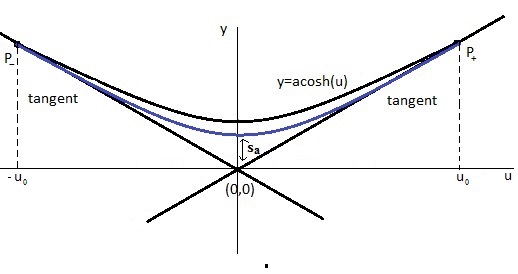}
      \caption{Shrinking procedure.}
      \label{shrink_gr}
      \end{figure}

Let $P_{\pm}= (\pm u_0, f_a(\pm u_0))$. We define a family of convex continuous functions $\{\bar f_{s_a}\}_{s_a\in(0,a]}$ by

\begin{equation*}
\bar f_{s_a}(u) = \left\{
  \begin{aligned} & f_a(u), \qquad &\text{ if }\quad &u\not\in(-u_0,u_0),\\
& Au^4+Bu^2+s_a, \qquad &\text{ if }\quad &u\in(-u_0,u_0),
\end{aligned} \right.
\end{equation*}
where $A = \frac{f'_a(u_0)u_0-2f_a(u_0)+2s_a}{2u_0^4}$ and $B = \frac{4f_a(u_0)-f'_a(u_0)u_0-4s_a}{2u_0^2}$. These functions coincide with $f_a$ outside of interval $(-u_0,u_0)$ and are smooth at every point except $P_{\pm}$. (See Figure~\ref{shrink_gr}, where $\bar f_{s_a}$ is shown in blue).

Now, let $\bar\delta>0$ be a sufficiently small number. By Proposition~\ref{prop:smoothing}, we can smooth the functions $\bar f_{s_a}$ in the $\bar\delta$-neighborhoods of the points $P_{\pm}$ and obtain   smooth functions $\tilde f_{s_a}$.

Consider the new metrics $\tilde\sigma_{s_a}$  that coincides with $\sigma_a$ outside of the collar $D$ and inside is  equal to 

\begin{equation}\label{rotation-metric}
du^2+ \tilde f_{s_a}^2(u)dv^2
\end{equation}

Notice that $\gamma$, i.e the image of the circle $u=0$, is a closed geodesic in this metric of length $2\pi s_a$ and hence can be made arbitrarily short. 
Denote the total area  of $M$ with respect to the metric $\tilde\sigma_{s_a}$ by  $V_{s_a}$.

Note that $V_{s_a}$ can be made arbitrarily close to $V$ if our construction begins from a metric of constant negative curvature far enough in the moduli space. When we say that $a$ tends to $0$, we mean that we apply construction to a family of metrics of constant negative curvature with arbitrarily short simple closed geodesics and the properties necessary for the construction. We denote by $\{g_{s_a}\}$ the resulting family of metrics normalized to have total area equal to $V$. These metrics have constant negative curvature $K_{s_a}$ outside the collar of length $u_0$ and have a short systole of length $a\sqrt{\frac{V}{V_{s_a}}}$. Observe that $K_{s_a}\to K$ as $a\to 0$.

\subsection{Estimation of metric entropy in Construction I}\label{metr_1}

Recall that $D$ is a piece of the modified collar of fixed length (not depending on $a$), whose area tends to $0$ as $a$ tends to $0$. 

By \cite{M81}, for any metric $g$ of negative curvature  with the curvature function $K_g$
\begin{equation*}
h^\lambda_{g}\geqslant\int_M\sqrt{-K_g}d\mu_{g}.
\end{equation*}

In particular, 
$h^\lambda_{g_{s_a}}\geqslant\int_{M\setminus D}\sqrt{-K_{s_a}}d\mu_{g_{s_a}} = \sqrt{-K_{s_a}}\mu_{g_{s_a}} (M\setminus D)$. The last expression can be made arbitrarily close to $\sqrt{-K}$ by taking the initial metric of constant negative curvature in construction I far enough in the moduli  space with a sufficiently short simple closed geodesic (that corresponds to the fact that $a$ is sufficiently small). By the Gauss-Bonnet theorem, $\sqrt{-K} = \left(\frac{4\pi(G-1)}{V}\right)^{\frac{1}{2}}$.

Therefore, for any $\delta >0$ and for the initial metrics of constant negative curvature far enough in the Teichm\"uller space, we have 
\begin{equation*}
\left(\frac{4\pi(G-1)}{V}\right)^{\frac{1}{2}}-\delta<h^\lambda_{g_{s_a}}<\left(\frac{4\pi(G-1)}{V}\right)^{\frac{1}{2}},
\end{equation*}
for all metrics from the family $\{g_{s_a}\}$ of metrics with arbitrarily short simple closed geodesics and having constant negative curvature outside of a region of fixed length around the collar. 

We recall that the starting metric is of constant negative curvature $K$. Therefore, its metric entropy and topological entropy are equal to $\sqrt{-K}$. Also, the metric entropy with respect to Liouville measure varies continuously for a smooth family of metrics. Let us fix any positive number $\delta$. Then, we can construct an infinite smooth family of metrics such that we get a continuous curve with an end point at the point $(\sqrt{-K},\sqrt{-K})$ on the graph of possible values for entropies (Figure~\ref{fig:ent_set}). Moreover, the curve lies between two lines: one is the line corresponding to metric entropy equal to $\sqrt{-K}$ and the other to the line corresponding to metric entropy equal to $\sqrt{-K}-\delta$. 
 
\subsection{Estimation of topological entropy in Construction I}\label{top_1}

We will now examine the change in the topological entropy under our systole shrinking construction. To get an estimate from below on the topological entropy, we will use \eqref{top-fundamental}.

\begin{figure}[H]
      \centering
      \includegraphics[height=2cm]{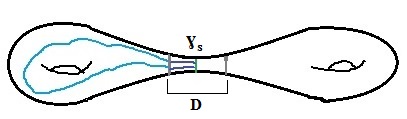}
      \caption{Curves which are used to estimate the topological entropy.}
      \label{shrink}
      \end{figure}

We fix an initial metric of constant negative curvature with simple closed geodesic of length $a$. Then, we apply the shrinking procedure to that geodesic to get a smooth family of metrics $\{g_s\}_{s\in(0,a]}$. Each metric $g_s$ has a simple closed geodesic $\gamma_s$ of length $s$. We can fix two points  $x,y\in\gamma_s$ and connect them by geodesics to the boundary of the modified region $D$. As the metric remains in the form \eqref{rotation-metric} in the  coordinate system above, we remain in normal coordinates for the new metrics $g_s$. Therefore, the parameter $u$ is still the arc length, and the geodesic connecting the two points to the boundary of $D$ have lengths bounded by a universal constant for the whole family of metrics $\{g_s\}$. Also, the points on the boundary of $D$ for every metric $g_s$ can be connected by a closed curve from a non-trivial homotopy class of length bounded by a universal constant depending on the initial metric of constant curvature. Let $\theta_s$ be a loop that travels along $\gamma_s$ connecting two chosen points, then along a geodesic to the boundary of $D$, then along an arc of a loop from a non-trivial homotopy class which connects the two points on the boundary of $D$, then along the other geodesic from the boundary of $D$ to $\gamma_s$, and finally back to the initial point on $\gamma_s$. See Figure~\ref{shrink} for reference. From the discussion above, it follows that there exists a constant $c$ which depends only on the initial metric of constant curvature such that the lengths of the curves $\{\theta_s\}_{s\in(0,a]}$ are bounded by $c$.  Consider the subgroup $\Pi$ of the fundamental group $\pi_1(M,x)$ generated by $\gamma_s$ and $\theta_s$. This subgroup obviously does not depend on $x$.  Notice that this subgroup has infinite index in the surface group $\pi_1(M,x)$ and hence is a free group (for a proof see e.g. \cite{blog}. Since it has two generators it is a free group with two generators  that we will denote by $\Gamma$ and $\Theta$, and hence by Grushko Theorem \cite{Gr40} the generators $\Gamma$ and $\Theta$ are free.  For every word $w$  from  the generators there is  for each $s$ a  corresponding loop  $l_s(w)$ consisting of powers of $\gamma_s$ and $\theta_s$. Different  loops represent different elements of the fundamental group. The shortest curves in the homotopy classes of such curves are shorter than the constructed curves, and they are geodesics. 

Now  fix a (large) number $R$. We will estimate from below the number $N_{g_s}(x, R)$ by counting only  elements of $\Pi$. It is sufficient we consider those words $w$ for which the length of the loop $l_s(w)$ is less than $R$. Moreover, we will only look at  words  where $\Gamma$ and $\theta$ appear in positive powers. Suppose the symbol $\Theta$ appears $i$ times and symbol $\Gamma$ $k$ times. There are $\binom{i+k}{k}$ different words of that kind. The length of the corresponding loop composed  from $\theta_s$ and $\gamma_s$  does not exceed $ic+ks$.
We need to count those words for which $ic~+~ks~\leqslant~R$. The number of such loops is bounded below by $\binom{\lfloor \frac{R}{2s} \rfloor + \lfloor \frac{R}{2c} \rfloor}{\lfloor \frac{R}{2s} \rfloor}$, which corresponds to the number of loops obtained when choosing $k=\lfloor \frac{R}{2s} \rfloor$ and $i=\lfloor \frac{R}{2c} \rfloor$. Using Stirling's formula $\sqrt{2\pi}n^{n+\frac{1}{2}}e^{-n}~\leqslant~n!~\leqslant~e n^{n+\frac{1}{2}}e^{-n}$, we obtain 

\begin{multline*}
\log \binom{\lfloor \frac{R}{2s} \rfloor + \lfloor \frac{R}{2c} \rfloor}{\lfloor \frac{R}{2s} \rfloor} \leqslant \log\left(\frac{e}{2\pi}\frac{\left(\lfloor \frac{R}{2s} \rfloor + \lfloor \frac{R}{2c} \rfloor\right)^{\lfloor \frac{R}{2s} \rfloor + \lfloor \frac{R}{2c} \rfloor+\frac{1}{2}}}{\lfloor \frac{R}{2s} \rfloor^{\lfloor \frac{R}{2s} \rfloor+\frac{1}{2}}\lfloor \frac{R}{2c} \rfloor^{\lfloor \frac{R}{2c} \rfloor+\frac{1}{2}}}\right) = \\ =\log\left(\frac{e}{2\pi}\sqrt{\frac{\lfloor \frac{R}{2s} \rfloor + \lfloor \frac{R}{2c} \rfloor}{\lfloor \frac{R}{2s} \rfloor\lfloor \frac{R}{2c} \rfloor}}\right)+\lfloor \frac{R}{2s} \rfloor\log\left(1+\frac{\lfloor \frac{R}{2c} \rfloor}{\lfloor \frac{R}{2s} \rfloor}\right) + \lfloor \frac{R}{2c} \rfloor\log\left(1+\frac{\lfloor \frac{R}{2s} \rfloor}{\lfloor \frac{R}{2c} \rfloor}\right)
\end{multline*}

Thus, one sees that

\begin{equation*}
h_{g_s}\geqslant \lim\limits_{R\rightarrow \infty}\frac{\log\binom{\lfloor \frac{R}{2s} \rfloor + \lfloor \frac{R}{2c} \rfloor}{\lfloor \frac{R}{2s} \rfloor}}{R} \geqslant \frac{1}{2c}\log\left(1+\frac{c}{s}\right).
\end{equation*}

\begin{rem} 
The fact that topological entropy goes to infinity as the length of the systole goes to zero in Construction~I also follows from the result by Besson, Courtois and Gallot \cite[Corollaire 0.6]{BCG_margulis}.
\end{rem}

Since for a fixed original metric of constant negative curvature  the metrics $g_s$  depend smoothly on $s$  this completes the proof of Theorem~\ref{thm:incr_top_entropy}.

\section{Construction II}\label{constr3}

In this section, we describe a procedure for varying metric entropy with respect to the Liouville measure while almost preserving the topological entropy starting from the metrics constructed in Theorem~\ref{thm:incr_top_entropy}.

\begin{thm}\label{thm:var_metr_pres_top}
Let $M$ be a closed  orientable surface of genus $G\geq 2$ and $V>0$. For any positive numbers $B>A>\left(\frac{4\pi(G-1)}{V}\right)^{\frac{1}{2}}$ and  a sufficiently small positive number $\delta$, there exists a smooth family of metrics $\{g_{s,t}\}$, where $(s,t)\in [0,1]\times[0,1]$, of negative curvature with total area $V$ such that the following holds.
\begin{enumerate}
\item\label{D1} $\left(\frac{4\pi(G-1)}{V}\right)^{\frac{1}{2}}-\delta< h^\lambda_{g_{s,0}}\leqslant\left(\frac{4\pi(G-1)}{V}\right)^{\frac{1}{2}}$ for every $s\in[0,1]$.
\item\label{D2} $h^\lambda_{g_{s,1}}<\delta$ for every $s\in[0,1]$.
\item\label{D3} $h_{g_{0,t}}<A$ for every $t\in[0,1]$.
\item\label{D4} $h_{g_{1,t}}>B$ for every $t\in[0,1]$.
\end{enumerate}
\end{thm}

Continuity of entropies with respect to the parameters and a standard topological argument that proves that there is no retraction of the cube onto its boundary gives the following.

\begin{cor}\label{coro} For every pair $(h,h^\lambda)\in[A,B]\times[\delta, \left(\frac{4\pi(G-1)}{V}\right)^{\frac{1}{2}}-\delta]$ there exists a pair $(s,t)\in [0,1]~\times~[0,1]$ such that $h_{g_{s,t}} = h$ and $h^\lambda_{g_{s,t}} = h^\lambda$.
\end{cor}

Since $A$ can be chosen arbitrary close to $\left(\frac{4\pi(G-1)}{V}\right)^{\frac{1}{2}}$, $B$ arbitrary large, and $\delta$ arbitrary small, Corollary~\ref{coro} implies Theorem~\ref{mainthm}.

\subsection{Construction II': Polyhedral approximation}\label{constr2}

In this section, we construct a family of metrics with pairs of values of entropies arbitrarily close to Side A in  Figure~\ref{fig:constr_set}. This family is the motivation for Construction II described in Section \ref{systole_and_polyhedral}. We will not provide a separate proof of Theorem~\ref{thm:decr_metr_entropy} for Construction II' as it will follow from the proof of Theorem~\ref{thm:var_metr_pres_top} (Sections \ref{metr_3} and \ref{top_3}) for Construction II.

\begin{thm}\label{thm:decr_metr_entropy}
Suppose $M$ is a closed  orientable surface of genus $G\geq 2$ and $V>0$. For any positive number $\delta$ there exists a smooth family of metrics $\{g_t^\delta\}_{t\in[0,1]}$ of negative curvature with total area $V$ such that $g_0$ is a metric of constant curvature and the following holds.
\begin{enumerate}
\item $h^\lambda_{g^\delta_1}<\delta$.\label{C-inequality}

\item $\left(\frac{4\pi(G-1)}{V}\right)^{\frac{1}{2}}\leqslant h_{g_t}<\left(\frac{4\pi(G-1)}{V}\right)^{\frac{1}{2}}+\delta$ for every $t\in [0,1]$. 
\end{enumerate}
\end{thm}

Let $g_0$ be a metric of constant negative curvature $K$ with total area $V$. For any compact Riemannian surface there exists a geodesic triangulation with arbitrarily small diameter. Fix a family of such triangulations $\mathcal T^d$ for the metric $g_0$ parametrized by their diameter $d$. 
In fact we do not need a family with diameters exactly equal to $d$ for every $d$ but rather a sequence with diameters going to zero. For convenience  we may also assume that the sides of triangles in $\mathcal T^d$ are of order $d$ as $d$ tends to $0$, i.e. there is a constant $C>0$ such that all sides of all triangles in $\mathcal T^d$ are longer than $Cd$. This can be achieved by fixing a triangulation $\mathcal T_0$, splitting each triangle in four by adding midpoints as vertices, and iterating this process.

We assign to each edge of the triangulation $\mathcal T^d$ the length of this edge for the metric $g_0$. Note that in a surface of constant non-positive curvature any three numbers satisfying  the triangle inequality appear as length of the sides of uniquely defined  (up to isometry) geodesic triangle. Furthermore, if the curvatures are bounded and the triangles are small, then the distortions are also small. Let us formulate this observation rigorously. 

\begin{lem}\label{lem:comparison}
Let $\Delta^d_1$ and $\Delta^d_2$ be triangles  in planes of non-positive curvature $K_1, K_2$ greater or equal than $K$ with corresponding sides $x$, $y$, $z$ of order $d$.

For any positive number $\delta$ and sufficiently small $d$ (depends on $\delta$, $K$), the following holds.
\begin{enumerate}
\item Let $\alpha_1$ and $\alpha_2$ be the angles between the sides $y, z$ in the triangles of curvature $K_1$ and $K_2$, respectively. Then,
$$(1-\delta)\alpha_2<\alpha_1<(1+\delta)\alpha_2.$$
\item\label{2} Let $v_1$ and $v_2$ be areas of the triangles $\Delta^d_1$ and $\Delta^d_2$, respectively. Then,
$$(1-\delta)v_2<v_1<(1+\delta)v_2.$$
\item For every $t,s\in[0,1]$, we have
$$(1-\delta)\text{dist}_2(ty ,sz)< \text{dist}_1(ty,sz) <(1+\delta) \text{dist}_2(ty,sz),$$
where $ty, sz$ are points on the sides $y$ and $z$, respectively, at distances $ty, sz$ from the vertex formed by the sides $y,z$ and $dist_1$, $dist_2$ are the distance functions in planes of curvature $K_1, K_2$, respectively.
\end{enumerate}
\end{lem}
\begin{proof}
The proof follows from the cosine laws in the hyperbolic and Euclidean planes and Taylor series expansion using the fact that the sides of $\Delta^d_1, \Delta^d_2$ are of order $d$.
\end{proof}

For every triangle in the triangulation $\mathcal T^d$, there is a unique (up to  an isometry) corresponding triangle in the hyperbolic plane of any curvature and in the Euclidean plane with the same lengths of the sides. 

We obtain a family of singular metrics $\{\bar g_t\},\,\,0\le t\le1$,  which are the result of gluing by isometries along the edges of the corresponding triangles in $\mathcal T^d$ of curvature $(1-t)K$ according to their incidence in $\mathcal T^d$. The metrics $\{\bar g_t\}$ are the standard polyhedral type metrics which are non-singular on edges and have conical singularities at the vertices of $\mathcal T^d$. An easy way to see that the metrics are smooth on the edges is to consider polar coordinates centered at vertices in each triangle of $\mathcal T^d$ and combine polar coordinates centered at the same vertex. It follows from the comparison theorem that if $K_t>K$ then we have conical points of negative curvature type at each vertex of $\mathcal T^d$, i.e., the angle at each vertex is larger than $2\pi$.

\begin{lem}\label{lem:smooth_conical}
Suppose $\bar g$ is a metric on $M$ with conic singularity with the total angle  $\alpha>2\pi$ at a point $q$ and constant non-positive curvature $\bar K$ in a neighborhood of $q$. For every sufficiently small positive $\bar\eps$, there exists a metric $g$ of non-positive curvature on $M$ that is smooth in the $2B\bar\eps$-ball (for the metric $\bar g$) centered at  $q$ and coincides with $\bar g$ outside of the $B\bar\eps$-ball of $q$, where $B$ is positive constant larger than $1$ and depends only on $\alpha$ and  $\bar K$. Moreover, the metric $g$ can be chosen in such a way that it depends smoothly on  $\alpha$ and $\bar K$.
\end{lem}

\begin{proof}
Let $B_{\eps}(q)$ be the closed ball of radius $\eps=\bar\eps/2>0$ centered at $q$ for the metric $\bar g$, and $S_{\eps}(q)$ be the boundary of the ball $B_{\eps}(q)$. Observe that the length $L_\eps$ of $S_{\eps}(q)$ in the metric $\bar g$ is greater than the length of a circle of radius $\eps$ in a plane of curvature $\bar K$. 

Consider normal coordinates $(r,\theta)$ centered at $q$, where $r$ is the distance from $q$ in the metric $\bar g$ and $\theta$ is the renormalized angle parameter such that it varies from $0$ to $2\pi$. By the construction, $\bar g$ is locally rotationally invariant. Therefore, the matrix of the metric $\bar g$ with respect to these coordinates has the form $\left(\begin{array}{cc} 1 & 0\\ 0 & \bar f^2(r)\end{array}\right)$.  For the metric $\bar g$, we have
\begin{equation}\label{conic}
\bar f(r)=\left\{
  \begin{aligned}
&\frac{\alpha}{2\pi}r \qquad &\text{if} \quad \bar K=0,\\ 
&\frac{\alpha}{2\pi}\left(-\bar K\right)^{-\frac{1}{2}}\sinh({-\bar K}^{\frac{1}{2}}r)\qquad &\text{if} \quad \bar K<0,
\end{aligned}
\right.
\end{equation}

Notice that the right-hand side of \eqref{conic} depends smoothly on $\alpha$ and $\bar K$. This is obvious for $\bar K<0$ and   easily follows from  the Taylor expansion for $\bar K=0$.

We  modify the metric $\bar g$ by replacing the conical singularity by a cap of constant negative curvature and smoothing the result by negative curvature around the boundary of the cap. The scheme is demonstrated in Figure \ref{fig:cone_const_smooth}. Therefore, we would like to find a negative number $k_\eps$ and radius $r_0$ of order $\eps$ such that  
\begin{itemize}

 \item The length of $S_{\eps}(q)$ coincides with the length of the boundary of a ball of radius $r_0$ in the hyperbolic plane of curvature $k_\eps$. 
 \item The numbers $k_\eps$ and $r_0$ are such that the modified metric is $C^1$ after gluing the cap of curvature $k_\eps$ and radius $r_0$ in the place of a ball $B_\eps(q)$.
 \end{itemize}

\begin{figure}
  \centering
  \includegraphics[height=5.4cm]{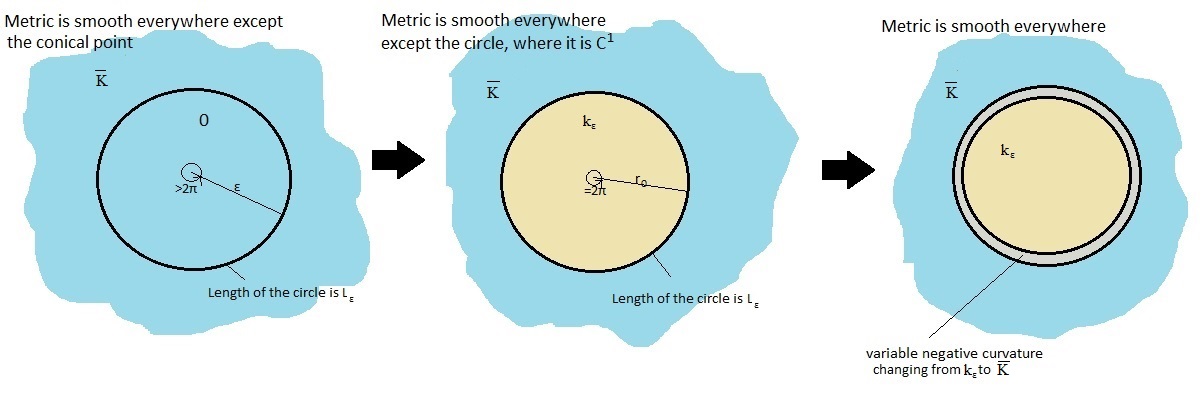}
    \caption{The smoothing procedure for conical points}
    \label{fig:cone_const_smooth}
\end{figure}

To make notations more compact let us denote $\frac{\alpha}{2\pi}$ by $b$. We now show that there exists a unique pair of numbers $(k_\eps, r_0)$ which satisfy the following system of equations.

\begin{equation}\label{parameter-choice}
\left\{
\begin{aligned}
&\frac{2\pi}{\sqrt{-k_\eps}}\sinh(\sqrt{-k_\eps}r_0) = L_\eps (=2\pi \bar f(\eps)) \hfill \\
&\cosh(\sqrt{-k_\eps}r_0) = b\cosh(\sqrt{-\bar K}\eps)(=\bar f'(\eps)) \hfill \end{aligned}
\right.
\end{equation}
The first equation determines the length of the boundary of a cap, the second guarantees  $C^1$-smoothness of the new metric.
From the \eqref{parameter-choice} and the hyperbolic trigonometric identity, we obtain 
\begin{equation}\label{glue-curvature}
1 = b^2\cosh^2(\sqrt{-\bar K}\eps)+\frac{L^2_\eps}{4\pi^2}k_\eps.
\end{equation} 
Therefore, 
\begin{equation}\label{curvature-glue}
k_\eps = -\frac{4\pi^2}{L^2_\eps}(b^2\cosh^2(\sqrt{-\bar K}\eps)-1).
\end{equation} 
Using the second equation in \eqref{parameter-choice} we obtain
\begin{equation}
r_0 = \frac{1}{\sqrt{-k_\eps}}\arccosh\left(b\cosh(\sqrt{-\bar K}\eps)\right) = \frac{L_\eps\arccosh\left(b\cosh\left(\left(-\bar K\right)^\frac{1}{2}\eps\right)\right)}{2\pi\left(b^2\cosh^2\left(\left(-\bar K\right)^\frac{1}{2}\eps\right)-1\right)^\frac{1}{2}}.
\end{equation}

Observe that since $b>1$, \eqref{curvature-glue} implies that  
$$k_\eps<0,\ \text{and}\, \,k_\eps\sim -\frac{b^2-1}{b^2\eps^2} = -C_b\eps^{-2}\,\text{as} \,\eps\to 0,$$ 
where $C_b$ is a positive constant which depends only on $\bar K$ and $\alpha$. Also
$$r_0\sim \frac{b\arccosh b}{\sqrt{b^2-1}}\eps, \,\, \text{as}\,\, \eps\to 0,$$ 
 Moreover, $ \frac{b\arccosh b}{\sqrt{b^2-1}}>1$ and  $\frac{b\arccosh b}{\sqrt{b^2-1}} \to1$ as $\alpha\to 2\pi$.

First, we obtain a $C^1$ metric in some neighborhood of  $q$ which coincides with $\bar g$ everywhere except some smaller neighborhood of  $q$. We replace $\bar f(r)$ by the following function $f(r)$:
\begin{equation}
f(r) = \left\{\begin{aligned}&\frac{1}{\sqrt{-k_\eps}}\sinh(\sqrt{-k_\eps}r) \qquad &\text{if} \quad &0\leqslant r\leqslant r_0\\
&\bar f(r-r_0+\eps)\qquad &\text{if} \quad &r> r_0.
\end{aligned}\right.
\end{equation}
The function $f(r)$ is convex. By Proposition~\ref{prop:smoothing}, we can smooth $f(r)$ in $\eps$-neighborhood of $r=r_0$ preserving the convexity. Moreover, this smoothing procedure can be performed to depend smoothly of $\alpha$ and $\bar K$. The resulting metric $g$ which we get using the smoothed $f(r)$ is smooth everywhere and coincides with $\bar g$ outside of the $(B_p+1)\eps$-ball (for the metric $g$) of $q$. Also, the metric $g$ has non-positive curvature everywhere as the curvature is equal to $-\frac{f''(r)}{f(r)}$, which is negative if $\bar K<0$ and non-positive if $\bar K=0$ by the preservation of convexity.
\end{proof}
 
Now we will use Lemma~\ref{lem:smooth_conical} to show how to get smooth metrics $\{g_t\}$ of non-positive curvature from $\{\bar g_t\}$ by smoothing conical points while preserving the essential properties for the estimates of entropies in the following sections. Let us fix one of our polyhedral metrics $\bar g_t$ with curvature $(1-t)K$ everywhere except the vertices of $\mathcal T^d$ where it has conical singularities of negative curvature type. For any sufficiently small positive $\bar \eps = o(d)$ as the diameter $d$ of $\mathcal T^d$ tends to $0$, we can apply Lemma~\ref{lem:smooth_conical} to every conical point of $\bar g_t$. As a result, we get a smooth metric $g_t$ of non-positive curvature. Applying the smoothing procedure for every metric $\bar g_t$, we get a smooth family  of metrics  that have  negative curvature if $t<1$ and non-positive curvature if $t=1$.   This  procedure depends smoothly on the curvature outside of the singular points and the angles at singular points and hence on the parameter $t$. 

Finally, let us normalize  the metrics to make the total area equal to $V$ and denote the resulting family by $\{g_t\}$. Notice that if $d$ is chosen  small enough we get from Lemma~\ref{lem:comparison}(\ref{2})  that the total area for the singular metrics $\bar{g_t}$  can be made  between $(1-\frac{\delta}{2})V$ and 
$(1+\frac{\delta}{2})V$.  If  $\epsilon$ is chosen  sufficiently small further modifications affect the total area arbitrary little. Thus one can assume that the normalization coefficient is between $1-\delta$ and $1+\delta$. 

%For a  more detailed discussion of  normalization see Section~\ref{top_3}.

\subsection{Combining shrinking of systole and polyhedral approximation}\label{systole_and_polyhedral}
Let $g_{s,0}$ be the family of  metrics built  by Construction I in the proof of Theorem~\ref{thm:incr_top_entropy} such that $h_{g_{1,0}}>2B$. Thus  statement \ref{D1} of Theorem~\ref{thm:var_metr_pres_top} holds. Statements \ref{D3} and \ref{D4} hold for $t=0$. We  extend this one-parameter family to a two-parameter family satisfying the rest of the statements via a construction similar to Construction II' in Section~\ref{constr2}. It is technically more involved and a certain care is needed to ensure smoothness. 

Recall that the metric $g_{s,0}$ has constant negative curvature $K_{s,0}$ outside of a collar $D$ of fixed length (which depends only on the curvature of the initial metric of constant curvature). Fix a family of triangulations $\mathcal T^d$ for the metric $g_{s,0}$ parametrized by their diameter with the same properties as in Section~\ref{constr2}. Also, we repeat the procedure of substituting a triangle in $\mathcal T^d$ of constant curvature $K_{s,0}$ outside of the $\bar\eps$-neighborhood of the collar $D$ by the triangle of constant curvature $K_{s,t}=K_s(1-t)$ with the same sides as the initial one. As a result, we obtain a family of singular metrics $\bar g_{s,t}$. The resulting metrics $\bar g_{s,t}$ have conical singularities (some vertices of $\mathcal T^d$) of negative type outside of the $\bar\eps$-neighborhood of $D$ and are only $C^0$ on two closed piecewise geodesic curves which connect the inserted triangles and the rest of the collar of $M$. All these procedures can be  performed to depend smoothly on the parameter $s$. We will omit dependence on $s$ throughout the rest of the construction, e.g. will write $g_0$ instead of $g_{s, 0}$ and $\bar g_t$ instead of $\bar g_{s,t}$.  
  
 \begin{figure}
      \centering
      \includegraphics[height=6cm]{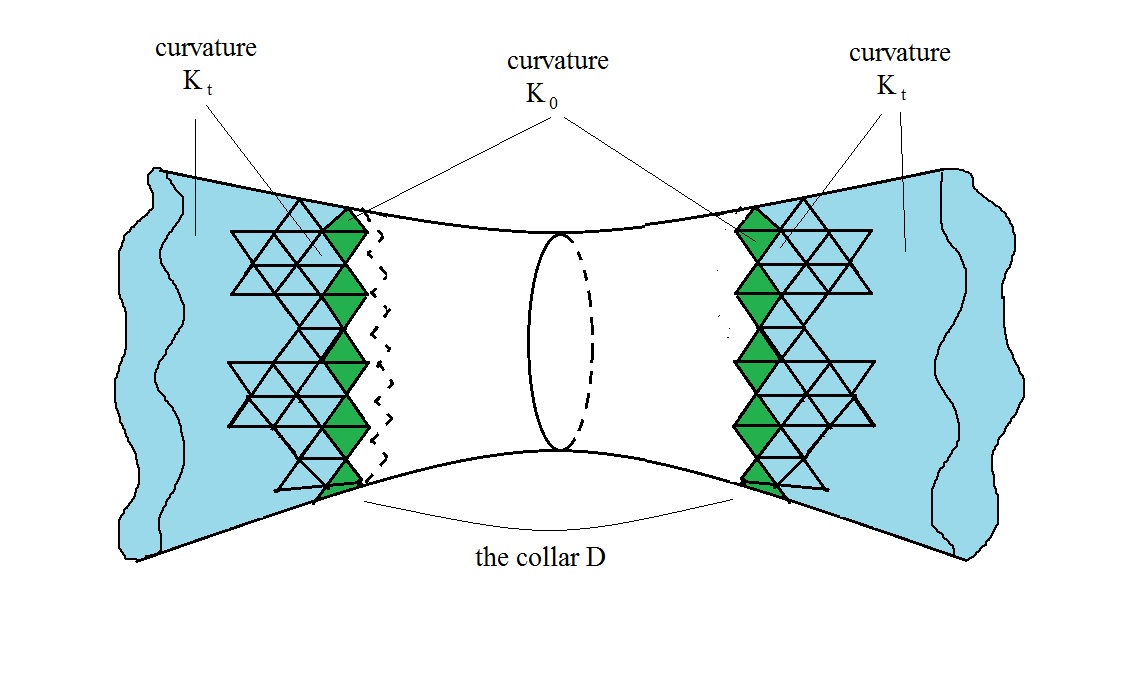}
      \caption{Combining shrinking of systole and polyhedral approximation.}
      \label{piecewise}
      \end{figure}

A method to obtain smooth metrics $\{g_t\}$ of non-positive curvature from $\{\bar g_t\}$ consists of a smoothing procedure for conical points and for closed piecewise geodesic curves such that the essential properties for the estimates of entropies in the following sections are preserved. First, we apply Lemma~\ref{lem:smooth_conical} to every conical point of $\bar g_t$ outside of $\bar\eps$-neighborhood of $D$.
Then, we proceed in two step with the smoothing procedure along closed piecewise geodesic curves. The first step is to apply Lemma~\ref{lem:smooth_conical_2} (the modified version of Lemma~\ref{lem:smooth_conical}) to the non-smooth points of the closed piecewise geodesics and sufficiently small $\bar\eps$. As a result, we obtain a metric with two closed piecewise geodesics which is smooth in a neigborhoods of the non-smooth points of them by the construction. The second step is to apply Lemma~\ref{lem:smooth_edge} to the maximal smooth pieces of the closed piecewise geodesics. We obtain a smooth metric $g_t$ of non-positive curvature. Applying the smoothing procedure for every metric $\bar g_t$, we get a smooth family of metrics that have negative curvature if $t<1$ and non-positive curvature if $t=1$. This procedure depends smoothly on the curvature outside of the singular point and closed piecewise geodesics and the angles at singular points and hence on the parameter $t$.

Finally, let us normalize  the metrics to make the total area equal to $V$ and denote the resulting family by $\{g_t\}$. Notice that if $d$ is chosen  small enough we get from Lemma~\ref{lem:comparison}(\ref{2})  that the total area for the singular metrics $\bar{g_t}$  can be made  between $(1-\frac{\delta}{2})V$ and $(1+\frac{\delta}{2})V$.  If  $\epsilon$ is chosen  sufficiently small further modifications affect the total area arbitrary little. Thus one can assume that the normalization coefficient is between $1-\delta$ and $1+\delta$.

\begin{lem}\label{lem:smooth_conical_2}
Suppose $\bar g$ is a metric on $M$ with conical singularity with the total angle $\alpha>2\pi$ at a point $q$. Further we assume that a neighborhood of $q$ is a union of two sectors of constant curvature $K_0$ and $\bar K>K_0$. For every sufficiently small positive $\bar\eps$, there exists a metric $g$ of non-positive curvature on $M$ that is smooth in the $B\bar\eps$-ball (for metric $g$) at $q$ coincides with $\bar g$ outside of the $\bar \eps$-ball at $q$, where $B>\frac{1}{2}$ and depends only on $K_0, \bar K$ and $\alpha$. Moreover, the metric $g$ can be chosen in such a way that it depends smoothly on $\alpha$, $K_0$ and $\bar K$.
\end{lem}

\begin{figure}
  \centering
  \includegraphics[height=6cm]{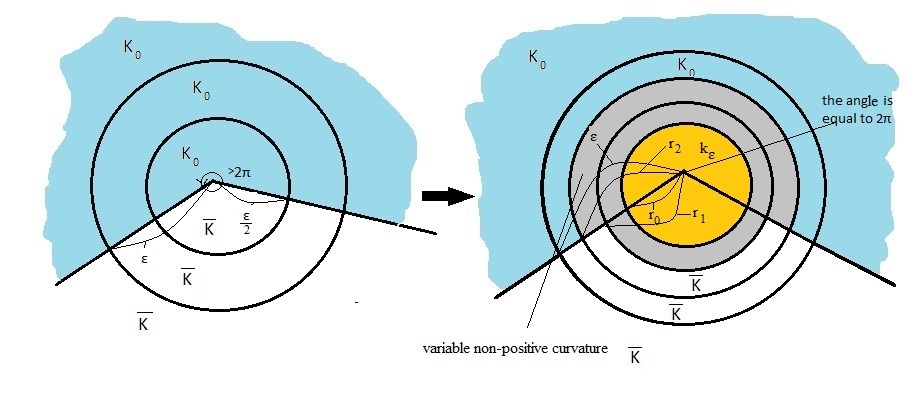}
    \caption{The smoothing construction in the case when two constant curvature meet}
    \label{fig:cone_smooth}
\end{figure}

\begin{proof}
Let $B_{\eps}(q)$ be a closed ball of radius $\eps=\bar\eps/2>0$ centered at $q$ for the metric $\bar g$. 

Consider normal coordinates $(r,\theta)$ centered at $q$, where $r$ is the distance from $q$ in the metric $\bar g$ and $\theta$ is the renormalized angle parameter such that it varies from $0$ to $2\pi$. The matrix of the metric $\bar g$ with respect to these coordinates has the form $\left(\begin{array}{cc} 1 & 0\\ 0 & \bar f^2(r, \theta)\end{array}\right)$. Let $\theta\in (0,\theta_0)$ correspond to the sector of curvature $K_0$ and other values of $\theta$ to the sector of curvature $\bar K$. We denote $\frac{\alpha}{2\pi}$ by $b$. For the metric $\bar g$, we have
\begin{equation}\label{two_curv_sector}
\bar f(r, \theta)=\left\{
  \begin{aligned}
&b(-K_0)^{-\frac{1}{2}}\sinh\left(\sqrt{-K_0}r\right) \qquad &\text{if} \quad \theta\in(0,\theta_0),\\	
&br \qquad &\text{if} \quad \theta\in(\theta_0, 2\pi) \text{ and } \bar K=0,\\ 
&b\left(-\bar K\right)^{-\frac{1}{2}}\sinh(\sqrt{-\bar K}r)\qquad &\text{if} \quad \theta\in(\theta_0, 2\pi) \text{ and } \bar K<0.
\end{aligned}
\right.
\end{equation}

Notice that the right-hand side of (\ref{two_curv_sector}) depends smoothly on $\alpha$, $K_0$ and $\bar K$.
 
As in Lemma~\ref{lem:smooth_conical} we modify $\bar g$ by replacing the conical singularity by a cap of constant negative curvature and smoothing the result by non-positive curvature around the boundary of the cap. The scheme is demonstrated in Figure~\ref{fig:cone_smooth}. 

Notice that for $\bar K < 0$, we have 
\begin{equation*}
b\left(-K_0\right)^{-\frac{1}{2}}\sinh\left(\sqrt{-K_0}r\right)>b\left(-\bar K\right)^{-\frac{1}{2}}\sinh\left(\sqrt{-\bar K}r\right)>br \text{ for every }  r>0.
\end{equation*} 

To guarantee that the modified metric has non-positive curvature, we need to preserve the convexity of $\bar f$ with respect to $r$. We proceed in the following way. First, we find a unique pair $(k_{\eps}, r_0)$ that satisfy
\begin{equation}
\frac{1}{\sqrt{-k_\eps}}\sinh\left(\sqrt{-k_\eps}r_0\right)=\frac{b\eps}{2} \qquad \text{and} \qquad \cosh\left(\sqrt{-k_\eps}r_0\right) = b.
\end{equation}
The first determines the length of the boundary of a cap, the second guarantees that we will be able to preserve the convexity of $\bar f$ with respect to $r$. 

The solution is the following.

\begin{equation}\label{sol_k_eps_r0}
k_\eps = \frac{4(1-b^2)}{b^2}\eps^{-2} \qquad \text{and} \quad r_0 = \frac{b\arccosh b}{2\sqrt{b^2-1}}\eps.
\end{equation}

Observe that since $b>1$, $k_{\eps}<0$. We can rewrite (\ref{sol_k_eps_r0}) as $k_\eps = -C_b\eps^{-2}$ and $r_0 = B_b\eps$, where $C_b$ and $B_b$ are positive constants which depend only on $b$, i.e. $\alpha$.  

First, we obtain a $C^0$ metric  in some neighborhood of $q$ which is smooth for $0\leqslant r<r_0$ and coincides with $\bar g$ everywhere except some small neighborhood of $q$. We replace $\bar f(r,\theta)$ by the following function $f(r, \theta)$.
\begin{equation*}
f(r, \theta) = \left\{\begin{aligned}&\frac{1}{\sqrt{-k_\eps}}\sinh(\sqrt{-k_\eps}r) \qquad &\text{if} \quad &0\leqslant r\leqslant r_0,\\
& \bar f\left(\frac{\eps}{2}+r-r_0, \theta \right)\qquad &\text{if} \quad &r> r_0, \quad \bar K=0,\quad \theta\in(\theta_0,2\pi),\\
& \bar f\left(\frac{\eps}{2}+r-r_1, \theta \right)\qquad &\text{if} \quad &r> r_1, \quad \bar K<0,\quad \theta\in(\theta_0,2\pi),\\
& \frac{b\eps}{2}+b(r-r_0)\qquad &\text{if} \quad &r_0<r< r_1, \quad \bar K<0,\quad \theta\in(\theta_0,2\pi),\\
& \bar f\left(\frac{\eps}{2}+r-r_2, \theta \right)\qquad &\text{if} \quad &r> r_2, \quad \theta\in(0,\theta_0),\\
& \frac{b\eps}{2}+b(r-r_0)\qquad &\text{if} \quad &r_0<r< r_2, \quad \theta\in(0,\theta_0),
\end{aligned}\right.
\end{equation*} 
where
\begin{equation*}
\begin{aligned}
r_1 &= r_0+b\left(-\bar K\right)^{-\frac{1}{2}}\sinh\left(\frac{(-\bar K)^{\frac{1}{2}}\eps}{2}\right)-\frac{\eps}{2} \quad \text{ if } \quad \bar K<0 ,\\
r_2 &= r_0+b\left(-K_0\right)^{-\frac{1}{2}}\sinh\left(\frac{(-K_0)^{\frac{1}{2}}\eps}{2}\right)-\frac{\eps}{2}.
\end{aligned}
\end{equation*}

\begin{figure}
  \centering
  \includegraphics[height=6cm]{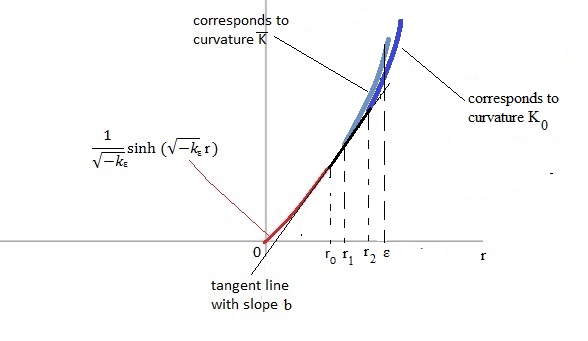}
    \caption{Sketch of the graph of the function $f(r,\theta)$ for different values of $\theta$}
    \label{fig:graph}
\end{figure}

The function $f(r,\theta)$ is convex in $r$ (see Figure~\ref{fig:graph}). By Proposition~\ref{prop:smoothing}, we can smooth $f(r, \theta)$ in $r$ in the $\eps/4$-neighborhood of $r=r_0$, in the $\eps/4$-neighborhood of $r=r_1$ if $\theta\in(\theta_0,2\pi)$ and $\bar K<0$, and in the $\eps/4$-neighborhood of $r=r_2$ if $\theta\in(0, \theta_0)$ preserving the convexity. Moreover, this smooth procedure can be performed to depend smoothly on $\alpha$, $K_0$ and $\bar K$. If $r>r_0$, then the resulting function $f(r, \theta)$ is non-smooth only for $\theta=0 (=2\pi)$ and $\theta = \theta_0$. For $\eps$ small enough, the resulting metric $g$ which we get using $f(r,\theta)$ coincides with $\bar g$ outside of the $\eps$-ball (for the metric $g$) at $q$. Moreover, $g$ has non-positive curvature where it is smooth and singularities of negative type, otherwise, by the preservation of convexity.
\end{proof}

\begin{lem}\label{lem:smooth_edge}
Suppose $\bar g$ is a metric on $M$ of non-positive curvature outside of singularities. Denote by $ab$ a geodesic piece with endpoints $a$ and $b$. Let $\bar\eps \ll$ the length of $ab$ for $\bar g$. Assume that the metric $\bar g$ is non-smooth on $ab$ outside of $\bar\eps$-neighborhoods of points $a$ and $b$ and smooth everywhere else in $\bar\eps$-neighborhood of $ab$. For every positive constant $\eps'<\bar\eps$, there exists a metric $g$ that is smooth metric of non-positive curvature inside $\bar\eps$-neighborhood of $ab$ and coincides with $\bar g$ outside of $\eps'$-neighborhood of the part of $ab$. 
\end{lem}

We apply Lemma~\ref{lem:smooth_edge} for the situation described on Figure~\ref{fig:pict_surf}. The metric $g$ can be chosen in such a way that it depends smoothly on $\alpha, K_0$ and $\bar K$.

\begin{figure}
  \centering
   \includegraphics[height=6cm]{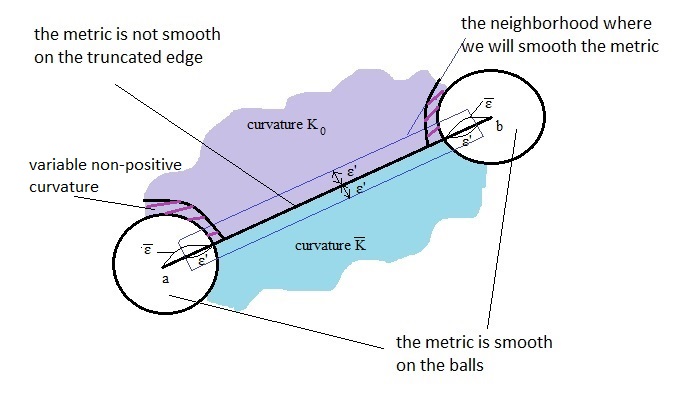}
    \caption{The picture on the surface around the edge}
    \label{fig:pict_surf}
\end{figure}

\begin{proof}
Let $p$ and $q$ be a points on $ab$ such that the metric $\bar g$ is non-smooth on $pq$ and smooth everywhere else on $\bar\eps$-neighborhood of $ab$.

Consider normal coordinates $(u,v)$ with respect to the geodesic $ab$, where $u$ is the distance to the geodesic $ab$ and $v$ is the arc length parameter along the geodesic. The matrix of the metric $\bar g$ with respect to these coordinates has the form $\left(\begin{array}{cc} 1 & 0\\ 0 & G(u, v)\end{array}\right)$, where $G(u,v)=f_1^2(u,v)$ on one side of $ab$ and $G(u,v) = f_2^2(u,v)$ on the other, and $f_1(0,v)=f_2(0,v)=1$ and $\frac{\partial}{\partial u}f_1(0,v)=\frac{\partial}{\partial u}f_2(0,v)=0$ for every $v$. The metric $\bar g$ has non-positive curvature what implies that the functions $f_1(u,v)$ and $f_2(u,v)$ are convex with respect to $u$. As a result, we have a picture similar to Figure~\ref{fig:pict_graph} for every fixed $v$.

\begin{figure}
  \centering
     \includegraphics[height=6cm]{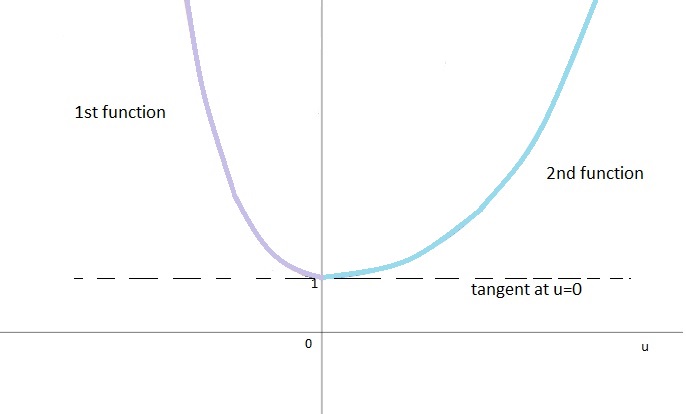}
    \caption{Function $\sqrt{G}$ for fixed value of $v$}
    \label{fig:pict_graph}
\end{figure}

Notice that  on $u=0$ (the geodesic piece $ab$), the function $\sqrt{G}$ is $C^1$ in $(u,v)$-coordinates but not $C^2$ on $pq$.  
Let $y_1(u,v) = \frac{\partial^2}{\partial^2 u}f_1(u,v)$ and $y_2(u,v) = \frac{\partial^2}{\partial^2 u}f_2(u,v)$. We define the following positive function $y(u,v)$ that is smooth everywhere except $u=0$ and $v$ values corresponding to $pq$. 

\[
y(u,v) = \left\{
  \begin{aligned}
    & y_1(u,v) \qquad& \text{ if } \quad &u\geqslant 0,\\
    & y_2(u,v) \qquad& \text{ if } \quad &u<0.
  \end{aligned} \right.
\]

By defintions, the functions $f_1(u,v), f_2(u,v)$ and $f(u,v)$ have the following expressions in terms of functions $y_1(u,v), y_2(u,v)$ and $y(u,v)$, respectively.

\begin{align*}
f_1(u,v) &= \int_0^u\int_0^{w_2}y_1(w_1,v)dw_1dw_2 + 1, \\
f_2(u,v) &= \int_0^u\int_0^{w_2} y_2(w_1,v)dw_1dw_2 + 1, \\
f(u,v)&:=\sqrt {G(u,v)} = \int_0^u\int_0^{w_2} y(w_1,v)dw_1dw_2 + 1.
\end{align*}

We obtain $g$ by replacing $f(u,v)$ by a smooth function $\hat f(u,v)$ which coincides with $f(u,v)$ outside of $\eps'$-neighborhood of $pq$and has non-negative second derivative in $u$ to guarantee non-positive curvature of $g$. In what follows, we construct $\hat f(u,v)$ with desired properties.

Let $\eps = \frac{\eps'}{2}$ and $\eps''$ be some positive constant. We define a smooth function $\tilde y(u,v) =\int_{\mathbb R^2}y(u-x, v-y)\theta_{\eps''}(x,y) \, dxdy$ that converges uniformly to $y$ on any compact set outside of the $\eps''$-neighborhood of $ab$ as $\eps''\rightarrow 0$, where $\theta_{\eps''}$ is a smooth positive kernel function on $\mathbb R^2$ equal to $0$ outside the $\eps''$-neighborhood of $0$. 

Let the point $p$ and $q$ have coordinates $(0,p_0)$ and $(0, q_0)$, respectively, and $0<p_0<q_0$. We define a smooth function $\hat y(u,v) = \phi(v)\tilde y(u,v)+(1-\phi(v))y(u,v)$ that coincides with $y(u,v)$ outside the $\eps$-neighborhood of $pq$, where $\phi(v)$ be a smooth mollifier function that is equal to $1$ for $v\in(p_0-\frac{\eps}{2}; q_0+\frac{\eps}{2})$ and $0$ for $v<p_0-\eps$ and $v>q_0$. Then, we substitute $f(u,v)$ by a smooth function $\hat f(u,v) = \psi(u)\tilde f(u,v) + (1-\psi(u)) f(u,v)$, where $\tilde f(u,v)=\int_0^u\int_0^{w_2}\hat y(w_1,v) \, dw_1dw_2 + 1$ and $\psi(u)$ is a smooth mollifier function that is equal $1$ for $|u|<\frac{\eps}{2}$ and $0$ for $|u|>\eps$.  
 
The second derivative with respect to $u$ of $\hat f(u,v)$ is non-negative for sufficiently small $\eps''$. It automatically holds for the points where $\hat f(u,v)=\tilde f(u,v)$. We show this property for all other points. The expression for the second derivative is the following.
\begin{align}\label{second_deriv_mod_f}
\frac{\partial^2}{\partial u^2}\hat f = \frac{\partial^2}{\partial u^2}f + \psi''(u)(\tilde f - f) + 2\psi'(u)\left(\frac{\partial}{\partial u}\tilde f - \frac{\partial}{\partial u}f \right) + \psi(u) \left(\frac{\partial^2}{\partial u^2}\tilde f - \frac{\partial^2}{\partial u^2}f \right).
\end{align}

Let $S$ be a compact set that is the closed $\eps$-neighborhood of $pq$ minus the open $\frac{\eps}{2}$-neighborhood of it. Then, $\tilde y$ converges uniformly to $y$ on $S$ as $\eps''\rightarrow 0$. Thus, $\hat y$ converges uniformly to $y$ on $S$ as $\eps''\rightarrow 0$. Therefore, we have the following sequence of limits on $S$ as $\eps''\rightarrow 0$.
\begin{align*}
&\left \| \frac{\partial^2}{\partial u^2}\tilde f- \frac{\partial^2}{\partial u^2}f \right \| = \left\| \hat y-y \right\| \rightarrow 0,\\
& \left \| \frac{\partial}{\partial u}\tilde f - \frac{\partial}{\partial u}f \right\| = \left\| \int_0^{u}\hat y(w_1,v)dw_1-\int_0^{u} y(w_1,v)dw_1\right\| \rightarrow 0,\\
&\| \tilde f-f \| = \left \| (\int_0^u\int_0^{w_2}\hat y(w_1,v)dw_1dw_2 + 1)-(\int_0^u\int_0^{w_2} y(w_1,v)dw_1dw_2 + 1)\right \| \rightarrow 0.
\end{align*}
As a result, we have that $\frac{\partial^2}{\partial u^2}\hat f$ is positive for sufficiently small $\eps''$ as $\frac{\partial^2}{\partial u^2}f$ is positive on $S$. 
\end{proof}

\begin{rem}
For $\bar K<0$, it is possible to modify the proof of Lemma~\ref{lem:smooth_conical_2} to obtain a negatively curved metric as a result of the smoothing procedures. The other way is to apply Proposition~\ref{prop:approx}. It follows from the expression of the curvature in $(u,v)$-coordinates and (\ref{second_deriv_mod_f}).
\end{rem}

\begin{rem}
If $\eps''$ is sufficiently small in comparison to $\bar\eps$, then we can guarantee that the curvature for the smoothed metric on the part of $pq$, where on one side we have $K_0$ and on the other $\bar K$, and its small translates in $u$-direction is bounded by a uniform constant $C_D$.
\end{rem}

\subsection{Estimation of metric entropy in Construction II}\label{metr_3}

Construction II gives us a smooth family of metrics $\{g_{s,t}\}$, where $(s,t)\in[0,1]\times[0,1]$, of non-positive curvature. By the choice of family of metrics $\{g_{s,0}\}$ from Theorem~\ref{thm:incr_top_entropy}, we can guarantee that $\left(\frac{4\pi(G-1)}{V}\right)^{\frac{1}{2}}-\delta<h^{\lambda}_{g_{s,0}}\leqslant \left(\frac{4\pi(G-1)}{V}\right)^{\frac{1}{2}}$ and that $h_{g_{0,0}}$ is less than any a priori given number which is larger than $\left(\frac{4\pi(G-1)}{V}\right)^{\frac{1}{2}}$ and that $h_{g_{1,0}}$ is larger than any a priori given number. We would like to point out that there is a family of such families of metrics with the properties above as we can initially choose the metric of constant curvature to have a sufficiently long neck. The choice will depend on $\delta$ and $A,B$ in Theorem~\ref{thm:var_metr_pres_top}. Metric entropy with respect to the Liouville measure and topological entropy vary continuously in a smooth family of metrics. Therefore, if we assume that we can show that $h^\lambda_{g_{s,1}}<\delta$ for every $s\in[0,1]$, $h_{g_{0,t}}<A$ and $h_{g_{1,t}}>B$ for every $t\in[0,1]$ and for every $h\in [A,B]$ there exists $s\in[0,1]$ such that $h_{g_{s,0}}=h$ for our constructed family of metrics then for every pair $(h,h^\lambda)\in[A,B]\times[\delta, \left(\frac{4\pi(G-1)}{V}\right)^{\frac{1}{2}}-\delta]$ there exists a pair $(s,t)\in [0,1]~\times~[0,1]$ such that $h_{g_{s,t}} = h$ and $h^\lambda_{g_{s,t}} = h^\lambda$. We now show we can guarantee the desired inequalities by appropriate choice of the initial family of metrics $\{g_{s,0}\}$ and proceeding with Construction II using sufficiently small neighborhoods for the smoothing procedures.

Let $g$ be one of the metrics in the family $\{g_{s,1}\}$. Recall that $g$ is the result of combination of Constructions I and II' and the smoothing procedure in $B\bar\eps$-balls at conical singularities and $\bar\eps$-neighborhoods of the closed piecewise geodesics $\gamma_1, \gamma_2$, where the collar $D$ from Construction I and the flat polyhedral metric are glued. From now on, $C$ will denote some constant (not fixed) which depends only on the triangulation $\mathcal T^d$ and $K$ (curvature of the initial constant curvature metric). It follows from \cite{W13} that the geodesic flow on $(M,g)$ is ergodic with respect to the Liouville measure.

Here we follow the methods in \cite{M81} for computing the metric entropy. Consider a geodesic $\gamma: \mathbb R\rightarrow (M, g)$ parametrized by arc length. Let $E:\mathbb R\rightarrow S^g M$ be one of the two perpendicular unit vector fields along $\gamma$. Then each perpendicular Jacobi field along $\gamma$ has the form  $y(t)E(t)$ for some function $y:\mathbb R\rightarrow\mathbb R$ which satisfies the Jacobi equation 
$$y''=-K_gy,$$ 
where $K_g$ is the Gaussian curvature along $\gamma$. The solution of the Jacobi equation is determined by initial conditions $y(a)$ and $y'(a)$, where $a\in\mathbb R$. Under the natural identifications of the horizontal and vertical components of vectors in the tangent bundle $T(S^gM)$ with a subbundle of the tangent bundle $TM$, we see that $y(a)E(a)$ and $y'(a)E(a)$ are the horizontal and vertical components of a vector $\xi\in T_v(S^gM)$ with $v=\gamma'(a)$. Consequently, $y(t)E(t)$ and $y'(t)E(t)$ are the horizontal and vertical components of $D(\phi^g_t)_v\xi$, which exhibits the relation between  Jacobi fields and the derivative of the geodesic flow. In particular, $\|D(\phi^g_t)_v\xi\|^2 = (y(t))^2+(y'(t))^2$. From the ergodicity of the geodesic flow with respect to the Liouville measure and Pesin's entropy formula, we obtain that for a typical (with respect to the Liouville measure) geodesic $\gamma$
\begin{equation}\label{entropy_Lyapunov_exponent}
h^{\lambda}_g = \lim\limits_{T\rightarrow\infty}\frac{\log\|D(\phi^g_T)_{\gamma'(0)}\xi\|}{T},
\end{equation} 
where we can take $\xi\in T_{\gamma'(0)}(S^gM)$ to be the vector specified above.

If $y$ is a solution of the Jacobi equation then the logarithmic derivative of $y$, $w=\frac{y'}{y}$ satisfies  the Ricatti equation 
\begin{equation}\label{Ricatti}w' = -K_g-w^2.\end{equation} 
This equation has a unique non-negative solution $w^+$ bounded for all positive and negative times that is equal to the geodesic curvature  $\kappa$ of the horocycle.\footnote{And there is also a unique non-positive solution  bounded for all times that is equal to the opposite of the geodesic curvature of the second horocycle, that is, the horocycle which $v$ points inside of.} Since geodesic flow is ergodic, the space average in \eqref{me-horocycle} is equal to the time average along  almost every orbit of the geodesic flow. Furthermore, any other  solution that is non-negative and bounded in positive time has the same asymptotic behavior  for $t\to\infty$ as $w^+$. Thus, the metric entropy can be computed by the formula 
\begin{equation}\label{entropy-Ricatti}
h^\lambda_g = \lim\limits_{T\rightarrow\infty}\frac{\int_0^Tw(s)ds}{T},
\end{equation}
where $w$ is any bounded non-negative solution of \eqref{Ricatti}.

Thus, any function $w$ satisfying \eqref{Ricatti} has critical points  when  $w=\pm \sqrt{-K_g}$,  monotonically increases between those bounds and monotonically decreases while  $w>\sqrt{-K_g}$ or  $w<-\sqrt{-K_g}$.  It follows that if
  \begin{equation}\label{initial-bounds} 0\leqslant w(0)\leqslant \max\limits_{M}\sqrt{-K_g},\end{equation} 
   then for all positive $t$  the solution $w(t)$ is non-negative and $w(t)\leqslant\max\limits_{M}\sqrt{-K_g}$. For the estimate of metric entropy from above it is enough to consider solutions satisfying \eqref{initial-bounds}. For the metric that we are now considering, the curvature $K_g$ is equal to $0$ when the geodesic is outside of the smoothed neighborhoods of conical points or the collar $D$, and otherwise it is less than $C^2\bar\eps^{-2}$ in absolute value if $\bar\eps$ is sufficiently small as the curvature on $D$ is determined by the initial metric from Theorem~\ref{thm:incr_top_entropy}. Consequently,  
\begin{equation}\label{Ricatti-curved}w(t)\leqslant C\bar\eps^{-1}.\end{equation}
 By solving the Ricatti equation on a segment $[s_1,s_2]$ where $K_g=0$, we obtain that for every $t\in[s_1,s_2]$ either $w(t)=0$ if $w(s_1)=0$ or 
\begin{equation}\label{Ricatti-flat}w(t)=\left(t-s_1+w^{-1}(s_1)\right)^{-1},\end{equation} otherwise.

To estimate the value of $h^{\lambda}_g$, we use both points of view given by \eqref{entropy_Lyapunov_exponent} and \eqref{entropy-Ricatti}.

% The {\em transverse measure} in $SM$ of the boundary of  the area  $\mathcal N$ of negative curvature coming from smoothed conical points is of order $\bar\epsilon$.  Hence   average return time to these regions is of the order of ${\bar\epsilon}^{-1}$.

Let $c_1, c_2$ be two boundary curves of the collar $F$ that are fully in the flat part (avoid the region $\mathcal N$ of negative curvature coming from the smoothed conical points) and such that the modified collar $D$ (with shrunk geodesic) and the gluing of it to the polyhedral part are inside $F$. There exists a constant $a$ (that depends only on the agreement on the lengths of the edges in $\mathcal T$) such that we can choose $c_1$ and $c_2$ that the distance to the region of negative curvature coming from the smoothed conical points and gluing of the collar $D$ is at least $a d$.

Consider a long typical geodesic segment $\{\gamma_t:\,\,0\leqslant t\leqslant T\}$. Without loss of generality we can assume that $\gamma_0, \gamma_T\in c_1\cup c_2$. Let $\{t_i\}_{i=1}^L$, $L\in\mathbb N$, be a sequence of times such that $t_L=T$, $\gamma_{t_i}\in c_1\cup c_2$ for all $i$, $\gamma_t\in M\setminus F$ if $t\in(t_{2j-1},t_{2j})$, and $\gamma_t\in F$ if $t\in(t_{2j},t_{2j+1})$, where $j\in\mathbb N$.

Notice that 
\begin{equation*}
\|D(\phi^g_T)_{\gamma'(0)}\xi\| = \|D(\phi^g_{t_1})_{\gamma'(0)}\xi\|\prod\limits_{i=2}^L\frac{\|D(\phi^g_{t_i-t_{i-1}})_{\gamma'(t_{i-1})}(D(\phi^g_{t_{i-1}})_{\gamma'(0)}\xi)\|}{\|D(\phi^g_{t_{i-1}})_{\gamma'(0)}\xi\|}
\end{equation*}

Therefore, if we can understand the behavior of $\frac{\|D(\phi^g_{t_i-t_{i-1}})_{\gamma'(t_{i-1})}(D(\phi^g_{t_{i-1}})_{\gamma'(0)}\xi)\|}{\|D(\phi^g_{t_{i-1}})_{\gamma'(0)}\xi\|}$ on $M\setminus F$ and $F$, then we get an estimate for $h^\lambda_g$.

To determine the expansion when the geodesic passes through the collar $F$, we use Gr\"{o}nwall's inequality applied to function $(y(t))^2+(y'(t))^2$ as 
\begin{equation*}
\frac{d}{dt}\left((y(t))^2+(y'(t))^2\right) = 2y(t)y'(t)+2y'(t)y''(t) = 2(1-K_g)y(t)y'(t)\leqslant (1-K_g)\left((y(t))^2+(y'(t))^2\right).
\end{equation*}
As a result, we have
\begin{equation}\label{expansion_collar}
\frac{\|D(\phi^g_{t_i-t_{i-1}})_{\gamma'(t_{i-1})}(D(\phi^g_{t_{i-1}})_{\gamma'(0)}\xi)\|}{\|D(\phi^g_{t_{i-1}})_{\gamma'(0)}\xi\|} = \left(\frac{(y(t_i))^2+(y'(t_i))^2}{(y(t_{i-1}))^2+(y'(t_{i-1}))^2}\right)^{\frac{1}{2}}\leqslant e^{\frac{1}{2}\int_{t_{i-1}}^{t_i}(1-K_g)dt}.
\end{equation}
Since geodesic flow is ergodic the space average of any continuous function is equal to the time average along a typical orbit of the geodesic flow. Therefore, by \eqref{expansion_collar}, we obtain that the total expansion coming from the collar $F$ is estimated in the following way as $T$ tends to infinity. 
\begin{multline}\label{input_collar}
\frac{1}{T}\sum\limits_{j=1}^{\lfloor\frac{L-1}{2}\rfloor}\log\frac{\|D(\phi^g_{t_{2j+1}-t_{2j}})_{\gamma'(t_{2j})}(D(\phi^g_{t_{2j}})_{\gamma'(0)}\xi)\|}{\|D(\phi^g_{t_{2j}})_{\gamma'(0)}\xi\|}\leqslant \frac{\sum\limits_{j=1}^{\lfloor\frac{L-1}{2}\rfloor}\int_{t_{2j}}^{t_{2j+1}}(1-K_g)dt}{2T}\rightarrow\frac{\int\limits_{F}(1-K_g)dA_g}{2}
\end{multline}
By taking the initial metric of constant negative curvature in Construction I far enough in the moduli space with a sufficiently short simple closed geodesic, we can guarantee that the $g$-area of $F$ and $\int_{F}(-K_g)dA_g$ (using the Gauss-Bonnet theorem) can be made arbitrarily small. As a result, the total expansion coming from the collar $F$ can be made sufficiently small by appropriate choice in Construction II of the initial family of metrics $g_{s,0}$ built by Construction I.

Now we estimate the total expansion coming from the area $M\setminus F$. 

Let $F_1$ be the $1$-neighborhood of $F$. If a segment $(t_{i-1}, t_{i})$ is such that $\gamma_t\in F_1$ for any $t\in(t_{i-1},t_i)$, then we again apply Gr\"onwall's inequality and \eqref{expansion_collar}. Therefore, we obtain that the total expansion coming from these pieces is controlled from above by 
\begin{equation}\label{input_1_nbd}
\frac{\int\limits_{F_1\setminus F}(1-K_g)dA_g}{2},
\end{equation}
that can be made arbitrarily small by taking the initial metric of constant negative curvature in Construction I far enough in the moduli space with a sufficiently short simple closed geodesic.

For a segment $(t_{i-1}, t_{i})$ such that there exists $t$ in that segment with $\gamma_t\notin F_1$ we use the following. 
\begin{multline}\label{expansion_hands}
\log\frac{\|D(\phi^g_{t_i-t_{i-1}})_{\gamma'(t_{i-1})}(D(\phi^g_{t_{i-1}})_{\gamma'(0)}\xi)\|}{\|D(\phi^g_{t_{i-1}})_{\gamma'(0)}\xi\|} = \log\frac{|y(t_i)|}{|y(t_{i-1})|}+\frac{1}{2}\log\frac{1+\left(\frac{y'(t_i)}{y(t_i)}\right)^2}{1+\left(\frac{y'(t_i)}{y(t_i)}\right)^2}=\\ = \int_{t_{i-1}}^{t_i}w(t)dt+\frac{1}{2}\log\frac{1+w(t_i)^2}{1+w(t_{i-1})^2}.
\end{multline}

We consider the following refinement of the partition $\{(t_i, t_{i+1})\}$ of $[0,T]$, where $i=1,\dots, L-1$. If $(t_i, t_{i+1})\notin F$, then we partition it into $L_i, R_i$, where $L_i$ is the maximal segment in $[t_i, t_{i+1}]$ in the zero curvature part with the end point $t_i$ and $R_i$ is the compliment. Otherwise, the segments is unchanged. Notice that each segment $L_i$ has length at least $a d$.

Denote the union of segments $\{R_i\}$ by $\mathcal R$ and the union of segments $\{L_i\}$ by $\mathcal L$. Let $N$ be the number of visits to $\mathcal N$, i.e. the number of maximal segments in $\mathcal R$ such that for each $t$ in the segment $\gamma_t\in\mathcal N$. Let $a_1,\dots, a_N$ be the left ends of those segments in the increasing order, $a_{N+1}=T,\,\, b_0=0$,  and $b_1,\dots, b_N$ the right ends of the segments also in the increasing order. Denote those segments $S_1,\dots,S_N$. 

%Let $a_1,\dots, a_N$ be the left ends of those segments in increasing order, $a_{N+1}=T,\,\, b_0=0$,  and $b_1,\dots, b_N$ the right ends of the segments also in increasing order. Denote those segments $S_1,\dots,S_N$.

\begin{lem}\label{fakevisits} $N<C_1T\bar\eps$, where the constant $C_1$ depends only on the triangulation $\mathcal T$.
\end{lem}
\begin{proof} Consider the discs of double radius concentric with discs of negative curvature. Let $\mathcal D$ be the union of those discs. Every segment $S_i,\,\,i=1,\dots,N$ lies inside a certain  disc of negative curvature and hence is  a part of segment inside the  concentric disc of double radius of length  greater than $2\bar\epsilon$. The total area of $\mathcal D$ is still of the order $\bar\epsilon^2$, hence, since $\gamma$ is a typical geodesic segment, by the Birkhoff Ergodic Theorem  the total length of the intersection $\gamma\cap\mathcal D$  is of order $\bar\epsilon^2 T$. Let $N'$ be the number of segments in that intersection of length greater than $2\bar\epsilon$. By the above arguments  $N\le N'<C_1T\bar\epsilon$. 
\end{proof}

Thus, we need to estimate from above 
\begin{equation}\label{eq1}
 \sum\limits_{R_k\in\mathcal R}\int\limits_{R_k}w(s) ds =\sum_{i=0}^N \int_{b_i}^{a_{i+1}}w(s) ds+\sum_{i=1}^N \int_{a_i}^{b_i}w(s) ds.
\end{equation}

We calculate the terms in the first sum using the explicit solution \eqref{Ricatti-flat} and estimate the terms of the second sum by the upper bound \eqref{Ricatti-curved} and Lemma~\ref{fakevisits}.

\begin{equation}\label{Ricatti-estimate1}
\sum\limits_{R_k\in\mathcal R}\int\limits_{R_k} w(s) ds\leqslant\sum_{i=0}^N\int_{b_i}^{a_{i+1}}(s-b_i+w^{-1}(b_i))^{-1} ds+C_2\bar\epsilon T\end{equation}

\begin{equation}\label{Ricatti-long}
\sum_{i=0}^N\int_{b_i}^{a_{i+1}}w(s)ds=\sum_{i=0}^N\log (a_{i+1}-b_i+ w^{-1}(b_i))-\log w^{-1}(b_i)
\end{equation}

Now we estimate each term in the right-hand part of \eqref{Ricatti-long}

\begin{align}
\log (a_{i+1}-b_i+ w^{-1}(b_i))-\log w^{-1}(b_i)=&\log(w(b_i)(a_{i+1}-b_i)+1)\\\le\log\left(C\bar\epsilon^{-1}(a_{i+1}-b_i)\left(1+\frac{\bar\epsilon}{C(a_{i+1}-b_i)} \right)\right)&\le\log C\bar\epsilon^{-1}+\log(a_{i+1}-b_i)+C_2\bar\epsilon.\nonumber
\end{align}

Summing over $i$, using the fact that $\sum_{i=0}^N (a_{i+1}-b_i)\le T$, convexity of the logarithm and the fact that a function $N\log\frac{T}{N}$ of $N$ is monotonically increasing if $eN<T$, we obtain an above estimate for the whole sum

\begin{align}\label{calculation2}
\sum_{i=0}^N\log (a_{i+1}-b_i+ w^{-1}(b_i))-\log w^{-1}(b_i)\leqslant-C_3\bar\epsilon\log\bar\epsilon T + \sum_{i=0}^N\log (a_{i+1}-b_i)+C_3\bar \epsilon^2T\\ \leqslant-C_4T\bar\epsilon\log\bar\epsilon + N\log\frac{T}{N} +C_3T\bar \epsilon^2 \leqslant-C_4T\bar\epsilon\log\bar\epsilon -CT\bar\epsilon\log\bar\epsilon +C_3T\bar \epsilon^2.\nonumber
\end{align}
 Combining \eqref{Ricatti-estimate1} and \eqref{calculation2} and assuming that $\bar\epsilon$ is small, we obtain 
 \begin{equation}\label{Ricatti-final}
\sum\limits_{R_k\in\mathcal R}\int\limits_{R_k}w(s) ds\leqslant-C_5T\bar\epsilon\log\bar\epsilon.
 \end{equation}

Let $N_c$ be the number of visits to $F$ outside of $1$-neighborhood of $F$, i.e. the number of maximal segments in the interval $[0,T]$ such that for each $t$ in the segment $\gamma_t\in F$. 

Recall that the size of the triangulation $\mathcal T$ is determined by the desired value of topological entropy and comparison lemma~\ref{lem:comparison}. If size $d_1$ of the triangulation is good for us, then any $d<d_1$ works. Therefore, without loss of generality we may assume that $d$ is the injectivity radius for the initial metric of constant curvature in Construction I. 

\begin{lem}\label{visits_collar} $N_c<C_c T d $, where the constant $C_c$ depends only on the total area $V$.
\end{lem}
\begin{proof} Consider the collar $F_1$ that is $1$-neighborhood of the collar $F$. Every segment $S_i,\,\,i=1,\dots,N_c$ lies inside $F$ and hence is a part of segment inside $F_1$ of length  greater than $1$. The total area of $F_1$ is still of the order $d$ if $\tilde g$ is far enough in the Teichm\"uller space, hence, since $\gamma$ is a typical geodesic segment, by the Birkhoff Ergodic Theorem  the total length of the intersection $\gamma\cap F_1$ is of order $d T$. Let $N_{c}'$ be the number of segments in that intersection of length greater than $1$. By the above arguments  $N_c\le N_{c}'<C_cTd $.   
\end{proof}

Denote by $l_k$ and $|L_k|$ the left endpoint and the length of $L_k$. Thus, we estimate the input of $\mathcal L$ using the explicit solution \eqref{Ricatti-flat}, the fact that the geodesic travels at least $a d$ before entering $L_k$, i.e. $w(l_k)\leqslant (a d)^{-1}$, convexity of the logarithm, the fact that a function $N_c\log\frac{T}{N_c}$ of $N_c$ is monotonically increasing if $eN_c<T$ and Lemma~\ref{visits_collar}.
\begin{multline}\label{from_collar}
\sum\limits_{L_k\in\mathcal L}\int\limits_{L_k} w(s) ds = \sum\limits_{L_k\in\mathcal L}\log(w(l_k)|L_k|+1)\leqslant\\\leqslant \sum\limits_{L_k\in\mathcal L}\left(-\log (a d)+\log|L_k|+\log(1+\frac{ad}{|L_k|})\right)\leqslant -N_c\log(a d) + N_c\log\frac{T}{N_c} +N_c\log 2\leqslant\\\leqslant -C_cTd\log(a d)-C_6 Td\log d +C_7dT< C_8Td\log d.
\end{multline}

Now we estimate the remaining part from the right hand side of \eqref{expansion_hands}.
\begin{equation}\label{visit_hands}
\sum\limits_{j=1}^{\lfloor\frac{L}{2}\rfloor}\frac{1}{2}\log\frac{1+w(t_{2j})^2}{1+w(t_{2j-1})^2}\leqslant \sum\limits_{j=1}^{\lfloor\frac{L}{2}\rfloor}\frac{1}{2}\log(1+w(t_{2j})^2)\leqslant -C_9Td\log d.
\end{equation}
In the above inequality, we took into account that the intersection with the area of zero curvature right before intersecting the collar is of length at least $a d$ and the number of visits of hands that get outside of $F_1$ is bounded by Lemma~\ref{visits_collar} . 

The formula \eqref{entropy_Lyapunov_exponent} for the metric entropy with respect to the Liouville measure and estimates \eqref{input_collar}, \eqref{input_1_nbd}, \eqref{Ricatti-final}, \eqref{from_collar} and \eqref{visit_hands}, we obtain that $h^\lambda_g$ can be made arbitrary small by the choice in Construction II of the initial family of metrics $g_{s,0}$ built by Construction I and by letting $\bar\eps$ be suffisiently small.

\subsection{Estimation of topological entropy in Construction II}\label{top_3}

In this section we demonstrate that we can guarantee Statements \ref{D3} and \ref{D4} of Theorem~\ref{thm:var_metr_pres_top} by showing $C^0$-closeness of $g_0$ and $g_t$ for every $t\in[0;1]$ by appropriate choice of parameters.

First, we show that $g_0$ and $\bar g_t$ are $C^0$-close by appropriate choice of a correspondence between the original triangles and their replacements and by choice of the size of the triangulation $\mathcal T^d$. Recall that $g_0$ and $\bar g_t$ coincide on the collar $D$ and differ outside it in the following way. Each triangle of curvature $K_0$ in $\mathcal T^d$ is replaced by a triangle of curvature $K_t$ with the sides of the same length as for $g_0$. Let $\Delta_1$ and $\Delta_2$ be the corresponding triangles with vertices $A_i, B_i$ and $C_i$, $i=1,2$, for $g_0$ and $\bar g_t$ in $\mathcal T^d$, respectively, with the following correspondence of the lengths of sides $|A_1B_1|_{g_0}=|A_2B_2|_{\bar g_t}$, $|A_1C_1|_{g_0}=|A_2C_2|_{\bar g_t}$, and $|B_1C_1|_{g_0}=|B_2C_2|_{g_t}$.  Now we construct a map between $\Delta_1$ and $\Delta_2$. The vertex $A_1$ is mapped to the vertex $A_2$. Let $P_1$ be a point inside $\Delta_1$ that differs from $A_1$. The $g_0$-geodesic containing $A_1$ and $P_1$ intersects $B_1C_1$ at a point $L_1$. Then, let $L_2$ be a point on $B_2C_2$ such that $|B_2L_2|_{\bar g_t}=|B_1L_1|_{g_0}$. Then, $P_1$ is mapped to the point $P_2$ such that $P_2$ belongs to the geodesic $A_2L_2$ and $\frac{|A_2P_2|_{\bar g_t}}{|A_2L_2|_{\bar g_t}}= \frac{|A_1P_1|_{g_0}}{|A_1L_1|_{g_0}}$. Under this map, $A_1$ is mapped to $A_2$, $B_1$ to $B_2$ and $C_1$ to $C_2$ and the following sides are identified by isometries: $A_1B_1$ with $A_1B_2$, $B_1C_1$ with $B_2C_2$ and $A_1C_1$ with $A_2C_2$. The constructed map is a homeomorphism. It is easy to show using polar coordinates centered at $A_1$ and $A_2$ for triangles $\Delta_1$ and $\Delta_2$, respectively, and Lemma~\ref{lem:comparison} that the constructed map is almost an isometry if the size of triangulation $\mathcal T^d$ is sufficiently small.

Second, we notice that $\bar g_t$ and $g_t$ are $C^0$-close if the size of regions where the metric $\bar g_t$ was smoothed is small enough. This observation follows from the estimates on the parameters and expression of $g_t$ in normal coordinates in the smoothing procedure.

As a result, $g_0$ and $g_t$ can be chosen arbitrary $C^0$-close by appropriate choice of parameters, which will also guarantee the closeness of total areas. Hence, topological entropies for $g_t$ and the metric which is its normalization that has area $V$ do not differ much. Moreover, the topological entropy is the exponential speed of the volume growth of balls in the universal cover. Therefore, sufficient $C^0$-closeness of $g_0$ and $g_t$ implies that their entropies are close.

\section {Further flexibility problems for metrics of negative curvature}\label{section:open}
\subsection{Detailed analysis of  flexibility for entropies}

One can consider the {\em entropy map} $H$ from the space $\mathcal M$  of smooth metrics of negative curvature and fixed area $V$ to $\mathbb R^2:\,\, H(g)=(h_g^\lambda, h_g)$. We proved that   the image of $H$ coincides with the set $\mathcal D$ from Figure~\ref{fig:ent_set}. In fact we proved that for any closed rectangle $\mathcal R\subset\mathcal D$ there exists a smooth map $I\colon \mathcal R\to \mathcal M$  such that the image of  $H\circ  I $ contains $ \mathcal R$. The  positive answer to the following conjecture would  provide a stronger statement 

\begin{conjecture} There exists a diffeomorphic embedding $I\colon\, \mathcal D\to \mathcal M$ such that 
$$H\circ I=Id.$$
\end{conjecture}

In order to prove this conjecture using a version of our construction  one needs, first, to extend  a single  two-parameter family from Section~\ref{constr3} to make the entropy map surjective, i.e to reach the values of entropy all the way to the sides A, B, and C  (see Figure~\ref{fig:constr_set}) and to infinity, and, second, to make sure that the entropy map remains injective. Stretching the image toward the sides B, C and to infinity requires only technical adjustments. Stretching toward the side A involves taking  finer triangulation and construction requires substantial modifications that nevertheless look feasible. Injectivity is a more difficult problem. The best hope to guarantee it  is to try to show that a  possible change of metric entropy during the shrinking systole process is less than the growth of topological entropy, and, similarly, a possible change of topological entropy during the polyhedral approximation is less than the decay of the metric entropy. 

\subsection{Entropy and conformal equivalence}

Theorem~\ref{mainthm} shows the flexibility of the pair of values for the topological and metric entropies of the geodesic flow on surfaces of negative curvature. A further question is if this flexibility remains when imposing additional natural restrictions. To prove Theorem~\ref{mainthm}, we modify a metric of constant negative curvature without preserving the conformal class of the initial metric. This is necessary for us to being able to go far enough in the Teichm\"uller space. Therefore, a natural question is whether a statement similar to Theorem~\ref{mainthm} holds when we additionally restrict the metric to a fixed conformal class or if new restrictions arise for entropies in a fixed conformal class. 

\begin{problem}\label{pr1}
Suppose $M$ is a closed  orientable surface of genus $G\geqslant 2$ and $V>0$. Let  $a, b$  be such that $a>\left(\frac{4\pi(G-1)}{V}\right)^{\frac{1}{2}}>b>0$. Does there exist a smooth metric $g$ of negative curvature in any conformal class such that $a = h_g$, $b = h_g^\lambda$ and $v_g = V$?
\end{problem}

Notice that large topological entropy requires a short systole \cite[Theorem 5.1]{BE}. Therefore, a good first step in approaching Problem~\ref{pr1} would be to answer the following question.

\begin{problem}\label{pr2}
Does there exist a positive lower bound on the length of the shortest nontrivial closed geodesic in a fixed conformal class for metrics of negative curvature on a closed surface with fixed total area?
\end{problem}

Although it is known from the uniformization theorem that any metric on a surface is conformally equivalent to a unique metric of constant curvature with the same total area, there is no easy way to check whether two given metrics lie in the same conformal class.

The next question is related to the construction in Section~\ref{constr2}, where we build polyhedral metrics with smoothed conical points which have topological entropy arbitrarily close to the topological entropy of a metric of constant curvature. Let $g$ be one of these metrics. Then, by the uniformization theorem 
\begin{equation}\label{Koebe}
g = e^{2u_g}\sigma_g
\end{equation}
 where $\sigma_g$ is a metric of constant curvature with total area equal to the total area for the metric $g$. The coincidence of total areas implies that $\int_Me^{2u_g}d\mu_{\sigma_g} = 1$. In particular, by the Cauchy-Schwarz inequality, $\int_Me^{u_g}d\mu_{\sigma_g} \leqslant 1$. Furthermore, the estimates from  \cite{K82} are actually stronger  than what  we used in Theorem~\ref{mainthm}, namely  
\begin{equation*}\label{eq-conformal} 
h_g\geqslant \frac{h_{\sigma_g}}{\int_Me^{u_g}d\mu_{\sigma_g}} \quad\text{and}\quad
h^{\lambda}_g \leqslant h_{\sigma_g}\int_Me^{u_g}d\mu_{\sigma_g}.
\end{equation*}
and both inequalities are strict unless $u\equiv 0$.
 Thus for any metric $g$ for which entropy values are close to sides A or B on Figure~\ref{fig:constr_set}
 the conformal coefficient 
  $\rho_g=\int_Me^{u_g}d\mu_{\sigma_g}$ is close to $1$.  Therefore, it is quite interesting to try to describe these functions $u_g$.

\begin{problem}\label{pr4}
Describe structurally the functions $e^{u_g}$, which come from the uniformization theorem applied to the metrics $g$ in Section~\ref{constr2}. Do they take values close to $1$ outside of a set of small measure if $h_g$ is close to $h_{\sigma_g}$? How do they look in the neighborhood of smoothed conical points?
\end{problem}

The answer to these questions will also give intuition on what restrictions on the values of entropies may or may not arise in a fixed conformal class.

 Problem~\ref{pr1} also has higher dimensional counterpart. There is a comparison theorem for topological and metric entropies for conformally equivalent metrics in any dimension  \cite{K82}. 
In particular, let $M$ be an m-dimensional manifold and $\sigma$ a Riemannian metric on $M$ of negative sectional curvature such that 
\begin {equation*}
h_\sigma=h^\lambda_\sigma. \,\footnote{The entropy conjecture \cite{K82} states that $\sigma$ must be a locally homogeneous metric.}
\end{equation*}
 Let $g=e^{mu}\sigma$  be a metric of negative curvature of the same total volume as $\sigma$,   i.e. $\int_Me^{mu}d\mu_\sigma=~1$ Let  $\rho_g=\int_Me^ud\mu_\sigma$ so that  $\rho_g<1$, unless $u\equiv 0$. Then
\begin{equation*}
h_g\geqslant(\rho_g)^{-1}h_\sigma\quad \text{and}\quad  h_g^\lambda\leqslant\rho_g h^\lambda_\sigma.
 \end{equation*}
\begin{problem}\label{pr3}
For any $a, b$ such that $a>h_{\sigma}>b>0$, does there exist a smooth metric $g$ of negative curvature and the same total volume conformally equivalent to $\sigma$  such that $a = h_g$, $b = h_g^\lambda$ and $v_g = V$?
\end{problem}

In higher dimension it is also natural to ask about flexibility of Lyapunov exponents with respect to the Liouville metric  along the line of \cite{BKRH} but  treatment of those questions is beyond the reach of current methods. 

\subsection{Flexibility beyond two entropies}

There are other important intrinsic characteristics of the geodesic flow on negatively curved surfaces beside  entropies $h_g$ and $h_g^\lambda$ that may be a natural subject of the flexibility analysis. Let us list some of those:
\begin{itemize}

\item  $\chi_g$: Positive Lyapunov exponent with respect to the measure of maximal entropy. It is closely related to the Hausdorff dimension of that measure.

\item  $h^H_g$: Entropy with respect to the harmonic invariant measure.

\item $\rho_g$: Conformal coefficient defined above.

\item $k_g=\int_{M}\sqrt{-K_g}d\mu_g$, where $K_g$ is the curvature function.

\end{itemize}
All of them are positive numbers. Let us summarize known relations among those quantities (\cite{M81}, \cite{K82}, \cite{L87}, \cite{R78}) assuming that the curvature of $g$ is not constant

\begin{equation}\label{inequalities-general}
k_g<h_g^\lambda < \left(\frac{4\pi(G-1)}{V}\right)^{\frac{1}{2}}< h_g^H<h_g<\chi_g, 
\end{equation}

\begin{equation}\label{inequalities-conformal}
 \rho_g>\max\{h_g^\lambda\cdot\left(\frac{4\pi(G-1)}{V}\right)^{-\frac{1}{2}}, \,(h_g^H)^{-1}\cdot\left(\frac{4\pi(G-1)}{V}\right)^{\frac{1}{2}}\}
 \end{equation}
 If any of the  inequalities in \eqref{inequalities-general} or \eqref{inequalities-conformal} is replaced by an equality, then all other become equalities and the metric $g$ has constant curvature (\cite{OS84}, \cite{K82}, \cite{LY85}, \cite{L90}). 
 
 The most general flexibility question is  hence the following.
 \begin{problem} Does any six-tuple of positive numbers, satisfying \eqref{inequalities-general} and \eqref{inequalities-conformal}, appear as  six characteristics in the order described above for a Riemannian metric of negative curvature on  the compact orientable surface of genus $G$  of total area $V$?
 \end{problem} 
 A positive solution  that  does not look improbable will require multi-parametric families of examples and deeper understanding of connections between  pertinent dynamical, PDE/variational and probabilistic properties of various structures related to Riemannian metrics. 
 
 Various subsets of this set of six characteristics correspond to partial flexibility problems. Thus, there is a wealth of questions connected with the flexibility of different characteristics for geodesic flows on surfaces of  negative curvature.
 Some of them look more accessible than others. For example, adding $k_g$ seems to be within reach. The main extra tool is an extension of the polyhedral approximation construction where not only vertices but edges are also brought into play. We plan to address this as well as some other cases in a subsequent paper.

\section{Appendix}

\subsection{Approximation  by metrics of negative curvature}

\begin{prop}\label{prop:approx}
Any smooth metric $g$ of non-positive curvature on a closed surface of genus $\geqslant 2$ can be $C^\infty$-approximated by a metric of negative curvature.
\end{prop}
\begin{proof}
    From the classical regularization theorem of Koebe (\cite{SS54}), it follows that $g = e^{2u}\sigma$ where $u$ is a smooth function for a unique metric of  constant negative curvature $\sigma$. Let $k$ be curvature for $\sigma$ and $K$ be the curvature for $g$. Then, we have $\Delta u - k + Ke^{2u} = 0$ where $\Delta$ is the Laplace operator for the metric $\sigma$. We rewrite this as $K = e^{-2u}(k-\Delta u)\leqslant 0$. Then, we need to find $u_1$ such that it is $C^\infty$-close to~$u$ and $e^{-2u_1}(k-\Delta u_1) < 0$. We may think of $u_1$ as $u+v$ so we need to find $v$ that is $C^\infty$-close to $0$ with $Ke^{2u}-\Delta v <0$. Define a smooth function $\rho$ with integral $0$ on $(M, g)$, i.e., let $\rho$ be non-positive inside components of negative curvature but $(Ke^{2u}-\rho)<0$ and such that on the flat parts it has value $\delta>0$ (small enough). Then, by theorems from the study of partial differential equations, we have that on a closed Riemannian surface there exists a smooth solution of the equation $\Delta v = \rho$, unique up to the addition of a constant. By rescaling $\rho$ by a smooth function $f(t)$ that converges to $0$ for $t\rightarrow\infty$, we obtain the desired result.
\end{proof}

Our procedure for obtaining metrics of non-positive curvature with the desired behavior comes from gluing various metrics together. Therefore, we need a statement showing that we can smooth convex functions while preserving convexity.

\begin{prop}\label{prop:smoothing}
Let $f:\mathbb R\rightarrow \mathbb R$ be a convex continuous function which is smooth everywhere except at a single point $P$. Then, for any $\delta>0$ there exists a smooth function $g$ such that it coincides with $f$ outside $\delta$-neighborhood of a point $P$.
\end{prop}
\begin{proof}
We will follow the ideas in \cite{G02}, but adapt for our case. Notice that convolution preserves convexity.

First, we discuss the case where the second derivative of $f$ is positive everywhere except the point $P$, where it does not exist.

Let us consider $\hat f(u):=\int_{\mathbb R}f(u-x)\theta_{\eps}(x)dx$, where $\theta_{\eps}$ is a smooth positive kernel function on $\mathbb R$ equal to 0 outside an $\eps$-neighborhood of 0. Therefore, by properties of convolution, we obtain that $\hat f$ is a smooth convex function. Furthermore, $\hat f$ uniformly converges to $f$ in $C^2$ norm in any compact set $K$ outside $\delta/3$-neighborhood of the point $P$ as $\eps\rightarrow 0$.

We define $g(u):=(1-\phi(u))f(u)+\phi(u)\hat f(u)$, where $\phi$ is a mollifier function which equals 0 outside of $\delta$-neighborhood of the point $P$ and 1 in $\delta/2$-neighborhood of the point $P$. Then, it is only needed to prove that the function $f$ is convex on the set $S$, where the function $\phi$ is not equal to $0$ or $1$. On $S$, we have
    $$g''(u)=f''(u)+2\phi'(u)(\hat f'(u)-f'(u))+\phi''(u)(\hat f(u)-f(u))+\phi(u)(\hat f''(u)-f''(u)).$$

    Therefore, we obtain $\|g''-f''\|\rightarrow 0$ as $\eps\rightarrow 0$ on $S$ from the uniform convergence on the compact sets outside $\delta/3$-neighborhood of the point $P$. So there exists $\eps<\delta/3$ such that $g''$ is positive everywhere as we have that $f''$ is positive.
		
    If $f(u)=au+b$, i.e. $f''(u)=0$ in the neighborhood not containing $P$, then we can put $g(u):=\hat f(u)$. Let $x$ be outside of $\eps$-neighborhood of $P$, $O_{\eps}(P)$, and $y$ in $\eps$-neighborhood of 0. So, $f(x-y)=a(x-y)+b$. Consequently, using facts that $\int_{\mathbb R}\theta_{\eps}(y)dy=1$ and $\theta_{\eps}$ is an even function, we obtain
    \begin{align*}
    \hat f(x) &= \int_{\mathbb R}[a(x-y)+b]\theta_{\eps}(y)dy = \\
    &= ax\int_{\mathbb R}\theta_{\eps}(y)dy-a\int_{\mathbb R}y\theta_{\eps}(y)dy+b\int_{\mathbb R}\theta_{\eps}(y)dy = \\
    &=ax\cdot 1-a\cdot 0+b\cdot 1 = f(x).
    \end{align*}
    for every $x\not\in O_{\eps}(P)$.
\end{proof}

\Addresses
\end{document}